\documentclass[11pt]{article}

\usepackage[letterpaper,margin=1in]{geometry}
\usepackage{amsmath}
\usepackage{amssymb}
\usepackage{mathtools}
\usepackage{amsthm}
\usepackage{graphicx}
\usepackage{booktabs}
\usepackage{multirow}
\usepackage{array}
\usepackage{algorithm}
\usepackage{algpseudocode}
\usepackage{url}
\usepackage{pgfplots}
\usepackage{microtype}
\usepackage[hidelinks]{hyperref}
\pgfplotsset{compat=1.18}
\providecommand{\doi}[1]{DOI~\href{https://doi.org/#1}{\nolinkurl{#1}}}

\DeclareMathOperator{\conv}{conv}
\DeclareMathOperator{\dist}{dist}
\DeclareMathOperator{\argmin}{argmin}
\DeclareMathOperator{\argmax}{argmax}
\newcolumntype{L}[1]{>{\raggedright\arraybackslash}p{#1}}
\newcommand{\R}{\mathbb{R}}
\newcommand{\ip}[2]{\left\langle #1,#2\right\rangle}
\newcommand{\norm}[1]{\left\lVert #1\right\rVert}
\newcommand{\eps}{\varepsilon}
\newcommand{\cA}{\mathcal{A}}
\newcommand{\cI}{\mathcal{I}}

\newcommand{\cV}{\mathcal{V}}
\newcommand{\one}{\mathbf{1}}

\theoremstyle{plain}
\newtheorem{theorem}{Theorem}[section]
\newtheorem{lemma}[theorem]{Lemma}
\newtheorem{proposition}[theorem]{Proposition}
\newtheorem{corollary}[theorem]{Corollary}
\newtheorem{assumption}[theorem]{Assumption}

\theoremstyle{definition}

\newtheorem{example}[theorem]{Example}

\theoremstyle{remark}
\newtheorem{remark}[theorem]{Remark}

\renewenvironment{proof}[1][Proof]{\par\noindent\textit{#1.}\ }{\par}

\begin{document}

\title{RA-DCA: A Randomized Active-Set DCA for Directional Stationarity in Max-Structured DC Programs}

\author{Yi-Shuai Niu\\
Beijing Institute of Mathematical Sciences and Applications (BIMSA), Beijing, China\\
\href{mailto:niuyishuai@bimsa.cn}{niuyishuai@bimsa.cn}\\
\href{https://orcid.org/0000-0002-9993-3681}{ORCID: 0000-0002-9993-3681}}

\date{}

\maketitle

\begin{abstract}
We study nonsmooth difference-of-convex programs whose subtracted convex term
is a finite maximum of smooth convex functions.  In this setting, standard DCA
iterations may converge to critical points that are not directionally
stationary, whereas exact active-vertex screening can be expensive when active
sets are large or combinatorial.  We propose RA-DCA, a vertex-first randomized
active-set DCA that projects active gradients onto sampled directions, checks a
sampled vertex residual, and uses a small linear program only as a
low-residual convex-combination fallback.  The method preserves the descent
structure of DCA and reduces the randomized screening layer to matrix
multiplications.  Under the stated regularity,
numerical active-set consistency, and random-embedding assumptions, every
accumulation point generated by the safeguarded method is directionally
stationary with probability one.  MATLAB experiments first test
the theorem on degenerate max-affine, max-quadratic, and sparse support-function
models, where the safeguard avoids nonstationary critical points and closely
tracks a full active-vertex scan.  Block top-k tests then show that the same
screening idea remains useful when exact aggregate enumeration is
combinatorial.  Trimmed-regression, complementarity, and QUBO diagnostics
separate cases where active-set selection helps from cases dominated by
multistart search, the DC split, or other problem-specific features.
\end{abstract}

\noindent\textbf{Keywords:} RA-DCA, max-structured DC programming, randomized active-set methods, directional stationarity, nonsmooth nonconvex optimization

\medskip

\noindent\textbf{MSC2020 subject classifications.} 90C26, 90C30, 65K05, 49J52, 68W20, 90C11, 90C33

\bigskip

\section{Introduction}

\subsection{Motivation and contributions}

Difference-of-convex (DC) programming is a standard modelling framework for
nonconvex optimization \cite{horst1999dc,lethi2005dc,lethi2018dc}.  In its
simplest unconstrained form one~minimizes
\begin{equation}\label{eq:problem}
        F(x) := g(x)-h(x), \qquad x\in\R^n,
\end{equation}
where both $g$ and $h$ are convex.  The DCA of Pham-Le Thi
\cite{pham1997dc,lethi2018dc} linearizes $h$ at the current iterate and
solves one convex subproblem.  This gives a robust and inexpensive algorithm,
but its natural limiting condition is criticality,
$\partial g(\bar x)\cap \partial h(\bar x)\ne\emptyset$, which can be weaker
than directional stationarity for nonsmooth DC objectives
\cite{pang2017computing}.

This paper focuses on the structured case
\begin{equation}\label{eq:maxh}
        h(x)=\max_{i\in\cI}\psi_i(x),
\end{equation}
where $\cI=\{1,\ldots,q\}$ is finite and each $\psi_i$ is smooth and convex.
This is also the basic block in finite sum-of-max models of the form
\[
        g(x)-\sum_{\ell=1}^p
        \max_{i\in\cI_\ell}\psi_{\ell i}(x).
\]
The single-max case isolates the active-set phenomenon without the additional
Cartesian product of block active sets.  Finite sums of such blocks are treated
through aggregate active vertices after the single-block analysis.
Max and sum-of-max structures appear in sparse and trimmed estimation
\cite{alfons2013sparse,bertsimas2016best}, piecewise-affine models
\cite{magnani2009convex}, exact penalties for complementarity-type
constraints \cite{luo1996exact,fletcher2004solving}, and robust optimization
\cite{bental2009robust,bertsimas2011theory}.  They also expose the limitation
of choosing an active subgradient without checking the active vertices.  At a
point where several $\psi_i$ are active, one vertex of the active
subdifferential may give a useful descent step while an averaged subgradient
may satisfy the weaker criticality condition and stop at a point that has an
immediate descent direction.
Section \ref{sec:problem} makes this distinction precise and gives a
one-dimensional example before the algorithm is introduced.

Several enhanced DC methods address this issue by considering active pieces
separately.  Examples include the algorithmic framework of Pang et al.
\cite{pang2017computing}, decomposition and statistical difference-max
extensions \cite{pang2018decomposition,cui2018composite}, and enhanced
proximal variants \cite{lu2019enhanced} and \cite{lu2019nonmonotone}.
Single-active-piece heuristics reduce some per-iteration costs, but their
practical behaviour is still tied to which active piece is selected.
In independent recent work, Le Thi, Huynh, and Pham Dinh
\cite{lethi2026unified} proposed a unified DCA-type framework for computing
directional stationary points of broad classes of DC programs, covering finite- and
infinite-maximum representations of the subtracted convex component and, more
generally, continuous convex components.  Their contribution is a general
convergence framework based on solving selected convex subproblems; its
randomized variant samples components to reduce the number of subproblems.
By contrast, the present paper takes a more specialized and computational
viewpoint: it focuses on finite
max-structured and sum-of-max DC models and uses randomized active-set
screening based on vertex residuals, random projections, and a
convex-combination LP fallback.  Section \ref{sec:numerics} tests this
active-set selection mechanism in large or combinatorial active sets.
The question motivating RA-DCA is whether randomized directions can summarize
the active subdifferential computationally while still safeguarding the
stronger directional-stationarity condition.

Our contribution is a vertex-first randomized active-set DCA.  The method
projects active gradients onto randomly sampled directions and first evaluates
a sampled vertex residual.  If this residual is large, RA-DCA takes the most
violating active vertex; only in the low-residual regime does it solve a small
linear program for convex-combination weights over the active set.  This
ordering is deliberate: the linear program is a centering mechanism for DCA
criticality, whereas the vertex residual is the directional-stationarity
safeguard that prevents accepting a merely averaged criticality certificate.

The resulting method has a simple computational kernel: its main additional
cost is forming products between a random direction matrix and an
active-gradient matrix.  This structure can be vectorized and is the basis
for the reported numerical implementation.  The main convergence theorem is
proved for the single finite maximum; the exact block corollary records what
must be safeguarded when finite sums of max terms are handled through
aggregate active vertices.  A supporting worst-iterate residual bound is proved
in Appendix \ref{sec:residualcomplexity}; it gives stationarity-complexity
context without changing the main asymptotic theorem.

The computational claims are correspondingly specific.  RA-DCA targets
iterations at which
several active pieces compete and an averaged active subgradient can hide a
descent direction.  In smooth portions of a trajectory, or in problems where
multistart exploration is the dominant issue, the randomized active-set rule
need not be the limiting factor.  The numerical section therefore separates
direct tests of the theorem from block aggregate tests and boundary penalty
diagnostics, so that performance gains are interpreted relative to the amount
of active-set ambiguity present.

\subsection{Model classes and applications}
\label{sec:intromodels}

The finite-max model \eqref{eq:maxh} is structured enough to permit a clean
active-set theory, but broad enough to cover several useful DC constructions.
Several model classes clarify what is covered directly by the present paper
and what is treated through the block extension.

\emph{Piecewise-affine and robust losses.}
If
\begin{equation*}
        h(x)=\max_{i\in\cI}(a_i^\top x+b_i),
\end{equation*}
then \(\partial h(x)\) is exactly the convex hull of the active affine
gradients.  This includes max-affine penalties, worst-case linear models,
polyhedral robust losses, and the signed-pair instances used in Section
\ref{sec:numerics}; see, for example,
\cite{magnani2009convex,bental2009robust,bertsimas2011theory}.  RA-DCA
applies directly, and the randomized projection step reduces to products
between the sampled direction matrix and the active gradient matrix.

\emph{Smooth finite maxima.}
The same framework applies when the pieces are smooth convex functions, for
example
\begin{equation*}
        h(x)=\max_i \{a_i^\top x+\tfrac{\gamma_i}{2}\norm{x}^2\}
        \quad\hbox{or}\quad
        h(x)=\max_i \psi_i(x).
\end{equation*}
The active gradients now change with \(x\), but the active subdifferential is
still the convex hull in \eqref{eq:subdiffh}.  The max-quadratic tests in
Section \ref{sec:numerics} are included to check this non-affine case; the
basic finite-maximum calculus follows standard convex analysis
\cite{rockafellar1970convex,clarke1983optimization}.

\emph{Sparse and trimmed optimization.}
Several feature-selection and robust-estimation penalties use order
statistics or top-\(k\) terms \cite{alfons2013sparse,bertsimas2016best}.  For
example, top-\(k\) absolute-value terms can be written as the maximum over
subsets of size \(k\),
\begin{equation*}
        \norm{x}_{(k)}=\max_{\substack{S\subseteq\{1,\ldots,n\}\\ |S|=k}}
        \sum_{i\in S}|x_i|.
\end{equation*}
After the usual smoothing or sign-pattern treatment away from zero, the
active sets correspond to tied top-\(k\) subsets.  RA-DCA is relevant here
because ties produce many active gradients and averaged choices can be weak.
Large-scale implementations would need implicit active-set generation rather
than explicit enumeration of all subsets.

\emph{Complementarity and infeasibility penalties.}
For a linear complementarity system of the form
\begin{equation*}
        0\le x \perp Mx+q\ge0,
\end{equation*}
violations of nonnegativity can be penalized by terms such as
\[
        \max\{-x_i,\,-(Mx+q)_i,\,0\}.
\]
Such penalties are max-affine locally and
therefore have exactly the active-set structure exploited by RA-DCA.  Exact
penalty and nonlinear programming treatments of complementarity constraints
are classical in MPEC and LCP computation
\cite{luo1996exact,fletcher2004solving}.  A full complementarity merit
function typically contains a sum of these max terms and possibly a
product-type complementarity penalty.  Thus a single max block is covered
directly by the main theorem, while a complete LCP penalty model belongs to
the block viewpoint formalized in Section \ref{sec:blockextension}.  Section
\ref{sec:binarycomp} uses the box-complementarity case
\(0\le x_i\perp 1-x_i\ge0\) to motivate a less symmetric block test.

\emph{Integer and mixed-integer penalties.}
Binary restrictions \(y_i\in\{0,1\}\) are often relaxed with penalties such
as \(y_i(1-y_i)\) on \(0\le y_i\le1\), or with piecewise-linear penalties for
distance to \(\{0,1\}\).  These formulations are naturally DC and can contain
max-affine pieces.  They connect with algebraic DC decompositions of
polynomials \cite{ahmadi2018dc,niu2024dcsos} and with DC cutting-plane
approaches for mixed-binary linear programs \cite{niu2025pdccut}; related
sparse and binary optimization models are discussed in
\cite{bertsimas2016best}.  They are promising applications for active-set DCA
because many binary variables can become nearly tied near fractional values.
Unconstrained binary quadratic programming is a particularly convenient test
case \cite{beasley1990orlibrary} because the quadratic part can be split as
\(Q=Q_+-Q_-\) with
\(Q_\pm\succeq0\), while the binary penalty supplies the max blocks.
Boolean-polynomial dynamical-system methods have also used Max-Cut and QUBO
benchmarks to test continuous relaxations of binary polynomial models
\cite{niu2022boolean}.
However, the most useful mixed-integer penalty models generally combine many
blocks and additional constraints.  They are computational block extensions of
the single-max analysis, whereas the main theorem covers the single-block
case.

\emph{Clustering and sum-of-max models.}
DC formulations of assignment-type clustering or \(k\)-median often contain
a sum of max terms, one for each data point
\cite{lethi2014clustering,bagirov2016clustering}.  The active set then has a
product or block structure.  The present algorithmic idea extends naturally:
construct active gradients blockwise, sketch the aggregate active-gradient
matrix, and safeguard vertices or block vertices.  Section
\ref{sec:blockextension} records the corresponding residual and explains how
the single-max analysis extends when the aggregate active vertices are checked.
Clustering requires application-specific subproblem solvers and data
structures, and is therefore outside the present computational scope.

These examples motivate the current focus on \eqref{eq:maxh}.  It is the
basic active-set block that appears inside more elaborate sparse,
complementarity, mixed-integer, and clustering models.  The paper therefore
separates the analysis into a single-block theory and an exact aggregate-vertex
corollary for finite block sums, with computational block experiments used to
illustrate the latter block viewpoint rather than extend the theorem.

\subsection{Organization}

Section \ref{sec:problem} formalizes the finite-max DC model and distinguishes
criticality from directional stationarity.  Section \ref{sec:algorithm}
introduces RA-DCA, including the vertex-first safeguard and the
convex-combination LP fallback.  Section \ref{sec:convergence} proves descent,
almost-sure
directional stationarity, and the exact block extension.  Section
\ref{sec:implementation} records the implementation choices relevant to
reproducibility.  Section \ref{sec:numerics} reports the MATLAB experiments,
and Section \ref{sec:conclusion} closes the main text.  The appendices contain
supporting material rather than additional main claims: Appendix
\ref{sec:residualcomplexity} gives a worst-iterate residual bound, Appendix
\ref{sec:appendixblock} records the approximate block safeguard and the
numerical-active-set variants, Appendix \ref{sec:quboappendix} gives the
additional QUBO diagnostics, and Appendix \ref{sec:lpbackendappendix} reports
the LP and safeguard diagnostics.

\section{Problem class and stationarity}
\label{sec:problem}

For $x\in\R^n$, define the active index set
\begin{equation*}
        \cA(x):=\{i\in\cI:\ \psi_i(x)=h(x)\}.
\end{equation*}
Since $h$ is a finite maximum of smooth convex functions,
\begin{equation}\label{eq:subdiffh}
        \partial h(x)=
        \conv\{\nabla \psi_i(x): i\in \cA(x)\}.
\end{equation}
When $g$ is differentiable, the classical DCA criticality condition is
\begin{equation}\label{eq:critical}
        \nabla g(x)\in \partial h(x).
\end{equation}
By contrast, $x$ is directional-stationary for \eqref{eq:problem} if
\begin{equation}\label{eq:dstationarydef}
        F'(x;d)=\ip{\nabla g(x)}{d}
        - \max_{i\in\cA(x)}\ip{\nabla\psi_i(x)}{d} \ge 0
        \quad \hbox{for all } d\in\R^n .
\end{equation}
The following residual is useful both theoretically and computationally:
\begin{equation}\label{eq:dres}
        R_{\rm d}(x):=
        \max_{i\in\cA(x)}
        \norm{\nabla g(x)-\nabla\psi_i(x)} .
\end{equation}

\begin{lemma}\label{lem:dres}
For \eqref{eq:problem}--\eqref{eq:maxh}, \(x\) is
directional-stationary if and only if \(R_{\rm d}(x)=0\).  The residual also
controls the convex criticality residual:
\begin{equation*}
        \dist\bigl(\nabla g(x),\partial h(x)\bigr)
        \le R_{\rm d}(x),
\end{equation*}
and the inequality can be strict.
\end{lemma}

\begin{proof}
Condition \eqref{eq:dstationarydef} is equivalent to
\begin{equation*}
        \ip{\nabla g(x)}{d}
        \ge \max_{i\in\cA(x)}\ip{\nabla\psi_i(x)}{d}
        \quad \hbox{for all } d\in\R^n .
\end{equation*}
The maximum is no larger than the left-hand side if and only if every active
linear form is no larger than the left-hand side.  Hence this is equivalent
to
\begin{equation*}
        \ip{\nabla g(x)-\nabla\psi_i(x)}{d}\ge 0
        \quad \hbox{for all } d\in\R^n,\ i\in\cA(x).
\end{equation*}
Fix an active index \(i\) and set
\(w_i=\nabla g(x)-\nabla\psi_i(x)\).  Since the last inequality holds for
every direction \(d\), it holds in particular for \(d=-w_i\), giving
\(-\norm{w_i}^2\ge0\).  Thus \(w_i=0\), and
\(\nabla g(x)=\nabla\psi_i(x)\) for every active \(i\).  This is precisely
\(R_{\rm d}(x)=0\).  Conversely, if \(R_{\rm d}(x)=0\), then all active
gradients equal \(\nabla g(x)\), so
\[
        \max_{i\in\cA(x)}\ip{\nabla\psi_i(x)}{d}
        =\ip{\nabla g(x)}{d}
        \quad\hbox{for all }d,
\]
and \eqref{eq:dstationarydef} follows.

For the distance estimate, \eqref{eq:subdiffh} shows that each active
gradient \(\nabla\psi_i(x)\) belongs to \(\partial h(x)\).  Therefore
\[
        \dist\bigl(\nabla g(x),\partial h(x)\bigr)
        \le
        \min_{i\in\cA(x)}
        \norm{\nabla g(x)-\nabla\psi_i(x)}
        \le R_{\rm d}(x).
  \]
Strictness is demonstrated by Example \ref{ex:onedtrap}.
\qed
\end{proof}

\begin{example}[A one-dimensional nonstationary critical point]\label{ex:onedtrap}
Let \(g(x)=\frac12 x^2\) and \(h(x)=\max\{x,-x\}=|x|\).  Then
\[
        F(x)=\frac12x^2-|x|
        =\frac12(|x|-1)^2-\frac12 .
\]
At \(x=0\), \(\nabla g(0)=0\) and
\(\partial h(0)=[-1,1]\).  Hence \(0\in\partial h(0)\), so \(x=0\) is a
DCA critical point and
\(\dist(\nabla g(0),\partial h(0))=0\).  It is not
directional-stationary, because the active gradients are \(1\) and \(-1\),
and \(R_{\rm d}(0)=1\).  Equivalently, \(F'(0;d)=-|d|<0\) for every
\(d\ne0\).

By contrast, \(x=1\) and \(x=-1\) are single-active points satisfying
\(\nabla g(1)=1=\nabla x\) and \(\nabla g(-1)=-1=\nabla(-x)\).  Thus
both points are directional-stationary; they are also the global minimizers.
If traditional DCA uses the centered subgradient \(v=0\in\partial h(0)\),
then its subproblem gives \(x^+=v=0\), so the method can remain at the
non-directional-stationary critical point.  Choosing an extreme active
gradient instead moves the iterate to one of the two directional-stationary
minimizers.
\end{example}

\begin{figure}[!ht]
\centering
\begin{tikzpicture}
\begin{axis}[
    width=0.72\linewidth,
    height=0.38\linewidth,
    axis lines=middle,
    xmin=-3.05, xmax=3.05,
    ymin=-1.05, ymax=2.25,
    xlabel={$x$},
    ylabel={$F(x)$},
    xtick={-3,-2,-1,0,1,2,3},
    ytick={-0.5,0,0.5,1,1.5,2},
    grid=both,
    minor tick num=1,
    tick label style={font=\scriptsize},
    label style={font=\small},
    every axis plot/.append style={line width=1.1pt},
]
\addplot[blue, domain=-3:0, samples=120] {0.5*x^2 + x};
\addplot[blue, domain=0:3, samples=120] {0.5*x^2 - x};
\addplot[only marks, mark=*, mark size=2.1pt, red!75!black]
    coordinates {(-1,-0.5) (1,-0.5)};
\addplot[only marks, mark=*, mark size=2.1pt, orange!85!black]
    coordinates {(0,0)};
\node[anchor=north east, font=\scriptsize, red!75!black, yshift=-1pt]
    at (axis cs:-1,-0.5) {$(-1,-1/2)$};
\node[anchor=north west, font=\scriptsize, red!75!black, yshift=-1pt]
    at (axis cs:1,-0.5) {$(1,-1/2)$};
\node[anchor=south west, font=\scriptsize, orange!85!black]
    at (axis cs:0,0) {critical non-d-SP};
\end{axis}
\end{tikzpicture}
\caption{The one-dimensional nonstationary critical point in Example \ref{ex:onedtrap}:
\(x=0\) is a DCA critical point but is not directional-stationary, while
\(x=\pm1\) are directional-stationary global minimizers.}
\label{fig:onedtrap}
\end{figure}
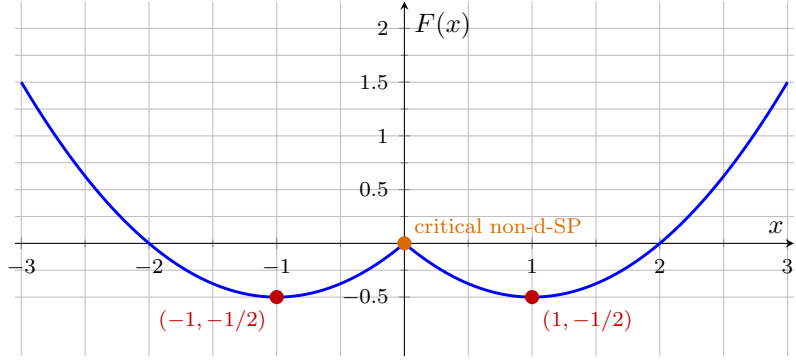

Lemma \ref{lem:dres} is the key structural distinction used by RA-DCA.
The criticality test asks whether \(\nabla g(x)\) lies in the convex hull of
the active gradients.  Directional stationarity asks for more: every active
vertex must give the same first-order response as \(\nabla g(x)\).  Thus an
averaged active subgradient may be a valid DCA certificate while still hiding a
descent direction exposed by an extreme active gradient.  RA-DCA separates
these two tasks: a convex combination is used for stable centering, while a
vertex-level residual provides the stationarity safeguard.
The condition is not vacuous: at points with a single active piece it reduces
to the usual smooth stationarity equation
\(\nabla g(x)=\nabla\psi_i(x)\).  At a genuine nonsmooth tie with different
active gradients, however, directional stationarity is possible only if those
gradients all agree with \(\nabla g(x)\).  This is exactly the point: for an
unconstrained objective \(g-\max_i\psi_i\), a nonzero active-gradient spread
creates a feasible descent direction, so a convex-hull criticality certificate
is not enough.

\section{The RA-DCA algorithm}
\label{sec:algorithm}

RA-DCA is a randomized active-set DCA method that keeps the usual convex DCA
subproblem while replacing a deterministic active-gradient choice by a
vertex-first randomized active-set calculation.
The convex-combination LP is retained as a centering fallback, but it is
called only after the sampled vertex residual is below the safeguard
threshold.  At iteration \(k\), the method first forms the numerical active
set
\begin{equation*}
        \cA_k=\{i\in\cI:\ h(x^k)-\psi_i(x^k)\le \eps_k\},
\end{equation*}
where $\eps_k\ge0$ is a numerical active-set tolerance.  Let
\begin{equation*}
        G_k = [\,\nabla\psi_i(x^k)\,]_{i\in\cA_k}\in\R^{n\times r_k},
        \qquad r_k=|\cA_k|,
\end{equation*}
and let $g_k=\nabla g(x^k)$.  A random direction matrix
$D_k\in\R^{m_k\times n}$ is generated by sampling rows uniformly on the unit
sphere and scaling by $\sqrt{n/m_k}$, or equivalently by using Gaussian rows
with variance $1/m_k$.  Thus $D_k z$ is a randomized sketch of $z$.

Algorithm \ref{alg:radca} summarizes the full iteration in a self-contained
form; the subsequent subsections define the safeguard residual, vertex rule,
LP fallback, and proximal DCA subproblem.

\begin{algorithm}[!ht]
\caption{RA-DCA}
\label{alg:radca}
\begin{algorithmic}[1]
\Require \(x^0\); tolerances \(\{\eps_k\}\), \(\{\tau_k\}\); sample-size rule
for \(m_k\in\mathbb N\); \(\sigma\ge0\)
\For{\(k=0,1,2,\ldots\)}
    \State Form
    \(\cA_k=\{i:\ h(x^k)-\psi_i(x^k)\le\eps_k\}\) and
    \(G_k=[\nabla\psi_i(x^k)]_{i\in\cA_k}\)
    \State Set \(g_k=\nabla g(x^k)\)
    \State Choose the number \(m_k\) of sampled directions
    \State Draw \(D_k\in\R^{m_k\times n}\)
    \State Compute
    \(\widehat R_k=\max_{i\in\cA_k}\norm{D_k(\nabla\psi_i(x^k)-g_k)}\)
    \If{\(\widehat R_k>\tau_k\)}
        \State Choose
        \(i_k\in\argmax_{i\in\cA_k}
        \norm{D_k(\nabla\psi_i(x^k)-g_k)}\)
        \State Set \(v^k=\nabla\psi_{i_k}(x^k)\)
    \Else
        \State Solve the active convex-combination LP for \(\alpha^k\)
        and set \(v_{\rm c}^k=G_k\alpha^k\)
        \State Set \(v^k=v_{\rm c}^k\)
    \EndIf
    \State Compute
    \(x^{k+1}=\argmin_x
    \{g(x)-\ip{v^k}{x-x^k}
    +\frac{\sigma}{2}\norm{x-x^k}^2\}\)
    \If{\(\norm{x^{k+1}-x^k}\) and the stationarity residual are below the
    requested tolerances}
        \State Stop
    \EndIf
\EndFor
\end{algorithmic}
\end{algorithm}

\subsection{Reference active-gradient rules}

Algorithm \ref{alg:radca} is the proposed method.  The numerical section also
uses three reference DCA rules that differ only in the way the active
linearization vector \(v^k\) is chosen; after that choice, they all use the
same convex update subproblem in Algorithm \ref{alg:radca}.  For the
single finite maximum, the rules are as follows.  Centered DCA uses the
active-gradient average
\[
        v^k_{\rm cen}
        =
        \frac{1}{r_k}\sum_{i\in\cA_k}\nabla\psi_i(x^k).
\]
Random-vertex DCA samples a single active piece,
\[
        v^k_{\rm rand}
        =
        \nabla\psi_{I_k}(x^k),
        \qquad I_k\sim\operatorname{Unif}(\cA_k).
\]
Full-vertex DCA uses the exact active-set residual scan,
\[
        v^k_{\rm full}
        =
        \nabla\psi_{i_k}(x^k),
        \qquad
        i_k\in\argmax_{i\in\cA_k}
        \norm{\nabla\psi_i(x^k)-g_k}.
\]
Thus centered DCA is a convex-hull based criticality rule.  Random-vertex DCA
tests the effect of choosing a single active piece.  Full-vertex DCA is the
exact active-set screening rule when the full residual scan is affordable.
RA-DCA can be read as a safeguarded randomized variant: it uses the sketched
residual to decide whether to take a violating active vertex, and uses the LP
candidate only when this safeguard does not trigger.

\subsection{Active convex-combination LP}

The fallback used when the vertex safeguard does not trigger is a convex
combination of active gradients.  For fixed \(x^k\), a DCA step is determined
by a linearization vector \(v^k\in\partial h(x^k)\).  Since \(h\) is a finite
maximum, this amounts to choosing simplex weights on active gradients:
\[
        v^k=G_k\alpha,\qquad \one^\top\alpha=1,\quad \alpha\ge0 .
\]
The LP \eqref{eq:lp} is introduced to choose this vector in a way that is consistent
with the DCA criticality condition \eqref{eq:critical}.  If \(x^k\) is
critical, then \(g_k\in\partial h(x^k)\), so there exists a simplex vector
\(\alpha\) satisfying \(G_k\alpha=g_k\).  Using such a vector makes the DCA
subproblem locally balanced at \(x^k\), whereas an arbitrary active vertex
may generate a spurious move even when DCA criticality already holds.
This also identifies the regime in which the LP can have a visible numerical
effect.  If the vertex residual is large, the safeguard deliberately rejects
the convex combination and the LP is not used.  Under the embedding condition,
when the safeguard does not trigger, all active vertex mismatches are already
small in the full space up to the sketching distortion.  The LP therefore
should not be expected to produce large objective improvements far from
stationarity.  Its role is local: it stabilizes the DCA step when the active
set contains nearly tied pieces, or when an \(\eps_k\)-active set includes
pieces that are inactive for the exact maximum but indistinguishable at the
chosen numerical tolerance.

Away from exact criticality, the same idea asks for a convex combination
\(G_k\alpha\) that is close to \(g_k\).  The full-space projection
\(\min_{\alpha\in\Delta}\|G_k\alpha-g_k\|\) is a simplex-constrained
quadratic problem.  RA-DCA does not solve this Euclidean projection.  Instead
it uses a randomized Chebyshev surrogate based on sampled directional
derivatives: it requires the mismatch to be small in the \(\ell_\infty\) norm
after applying \(D_k\), namely \(D_k(G_k\alpha-g_k)\).  The two formulations
coincide only in special
cases, for example if the full-space norm is replaced by \(\ell_\infty\) and
the sketch is \(D_k=I\).  Thus the LP is a centering device for selecting a
stable DCA linearization, not a certificate of directional stationarity; the
latter is checked by the vertex safeguard.

The formulation therefore chooses $\alpha\in\R^{r_k}$ in the simplex and the
smallest scalar \(t\) that bounds the sampled directional mismatch in the
\(\ell_\infty\) norm:
\begin{equation}\label{eq:lp}
\begin{array}{ll}
 \displaystyle \min_{\alpha,t} & t\\[1mm]
 \hbox{s.t.} &
        -t\one \le D_k(G_k\alpha-g_k)\le t\one,\\
      & \one^\top \alpha=1,\quad \alpha\ge0,\quad t\ge0.
\end{array}
\end{equation}
The LP has only $r_k+1$ variables and $2m_k+1$ affine constraints.  Its
coefficient matrix is
\begin{equation*}
        D_kG_k =
        \bigl[\ip{d_j}{\nabla\psi_i(x^k)}\bigr]_{j,i},
        \qquad
        D_kg_k =
        \bigl[\ip{d_j}{\nabla g(x^k)}\bigr]_j,
\end{equation*}
which is a matrix-matrix product plus a matrix-vector product.  This is the
part of the method that can be accelerated directly on a GPU.
Except in degenerate cases, such as a singleton active set, this Chebyshev
approximation problem over the simplex does not have a useful closed-form
solution; it is solved as a small continuous LP.

When the vertex safeguard is not triggered, let
\(\alpha^k\) solve \eqref{eq:lp} and define
\begin{equation*}
        v_{\rm c}^k=G_k\alpha^k\in\conv\{\nabla\psi_i(x^k):i\in\cA_k\}.
\end{equation*}
This vector is the active convex-combination candidate.  It is useful for a
stable DCA step, but it is not sufficient for directional stationarity.

\subsection{Vertex safeguard}

The reason for the safeguard is the distinction in Lemma \ref{lem:dres}.
The LP \eqref{eq:lp} can make a convex combination \(G_k\alpha\) close to
\(g_k\), which is the right test for DCA criticality.  This can nevertheless
hide a directional-stationarity violation: different active gradients may
cancel in their convex hull even though some active vertex is far from
\(g_k\).  Directional stationarity rules out this cancellation and requires
all active vertices to agree with \(g_k\).  The safeguard therefore checks
the active vertices themselves.

The sampled directional-stationarity residual is
\begin{equation}\label{eq:sampledres}
        \widehat R_k =
        \max_{i\in\cA_k}\norm{D_k(\nabla\psi_i(x^k)-g_k)} .
\end{equation}
It is a randomized proxy for the vertex residual \(R_{\rm d}(x^k)\).  When
this residual is large, using the convex-combination candidate \(v_{\rm c}^k\)
would risk accepting a merely critical point.  RA-DCA instead selects the
most violating active vertex,
which forces the next DCA subproblem to respond to the detected violation.
If the exact residual \(R_{\rm d}(x^k)\) can be scanned cheaply, then the
corresponding exact full-vertex rule is natural: choose
\[
        i_k\in\argmax_{i\in\cA_k}
        \norm{\nabla\psi_i(x^k)-g_k}
\]
and set \(v^k=\nabla\psi_{i_k}(x^k)\) whenever this maximum is nonzero.  This
rule is the full-vertex DCA baseline in Section \ref{sec:numerics}, and it
is the ideal active-set screening rule in the single finite-maximum case.
For an explicit single maximum with all active gradients already materialized,
this exact scan costs \(O(nr_k)\) and is usually no more expensive than a
dense sketched scan; setting \(D_k=I\) in the safeguard recovers precisely this
full-vertex rule.  In such explicit small- or medium-scale cases, the
full-vertex scan is the appropriate reference.  The randomized safeguard is
aimed at settings in which only directional products \(D_kG_k\) are formed,
when the same projected matrix is reused by the LP, or when the exact
aggregate-vertex scan in a sum-of-max model is combinatorial.  Moreover, when
the sampled residual is already below \(\tau_k\), the identity of the largest
sketched vertex is no longer a robust stationarity signal; a
convex-combination step is then used as a stable DCA centering step.  The
convergence proof uses \(\tau_k\downarrow0\), so this LP branch can persist
near an accumulation point only when the exact vertex residual is also small.
If $\widehat R_k$ is larger than a tolerance $\tau_k$, RA-DCA selects the
most violating active vertex
\begin{equation}\label{eq:vertexchoice}
        i_k\in\argmax_{i\in\cA_k}
        \norm{D_k(\nabla\psi_i(x^k)-g_k)}
\end{equation}
and sets $v^k=\nabla\psi_{i_k}(x^k)$.  In this case the LP
\eqref{eq:lp} is not needed.  Otherwise RA-DCA solves \eqref{eq:lp} and sets
$v^k=v_{\rm c}^k$.  This ordering is computationally important for large
active sets: the LP is used only when the sampled vertex residual is already
below the safeguard tolerance.  The next iterate is the proximal DCA point
\begin{equation}\label{eq:update}
        x^{k+1}=\argmin_x
        \left\{
        g(x)-\ip{v^k}{x-x^k}
        +\frac{\sigma}{2}\norm{x-x^k}^2
        \right\}.
\end{equation}
When $g$ is already strongly convex, $\sigma$ may be set to zero.  Otherwise
$\sigma>0$ makes \eqref{eq:update} strongly convex and controls the step.
Under this condition the minimizer in \eqref{eq:update} is unique.  The
integer \(m_k\) is the number of sampled directions at iteration \(k\).  It
may be fixed in advance or chosen from the active set after \(G_k\) is formed;
the embedding-based choice used in the analysis is given in
\eqref{eq:mkbudget}.

The two tolerances have different roles.  The parameter \(\eps_k\) defines the
numerical active set.  It should be large enough to include pieces that are
tied up to evaluation error, but small enough that inactive pieces do not
dominate the DCA linearization.  The convergence analysis assumes a
summable sequence, for instance
\(\eps_k=\eps_0/(k+1)^{1+\beta}\) with \(\beta>0\).  The parameter
\(\tau_k\) is the safeguard threshold: values above \(\tau_k\) force an active
vertex step, while values below \(\tau_k\) allow the LP-centered step.  The
almost-sure result requires \(\tau_k\downarrow0\), and the residual bound in
Appendix \ref{sec:residualcomplexity} shows that
\(\tau_k=O((k+1)^{-1/2})\), or a faster decay, is consistent with the
\(O(k^{-1/2})\) worst-iterate stationarity scale.  In finite-precision
experiments we use fixed machine-level active-set and safeguard tolerances as
finite-horizon approximations of these asymptotic schedules; these values are
chosen on the same scale as the reported stationarity and step tolerances.

\section{Convergence analysis}
\label{sec:convergence}

\subsection{Basic assumptions and descent}
\label{sec:descentanalysis}

The analysis follows the numerical active sets used in Algorithm
\ref{alg:radca}.  This is slightly different from the exact-active-set DCA
argument because a gradient selected from an \(\eps_k\)-active piece need not
belong to the exact subdifferential \(\partial h(x^k)\).  The descent estimate
therefore carries an \(\eps_k\) error term.  The second technical point is
active-set consistency at accumulation points: a piece active at the limit can
be inactive along a one-sided sequence approaching that limit, so the theorem
requires the numerical active sets to capture such limiting active pieces.

\begin{assumption}\label{ass:basic}
The following conditions hold.
\begin{enumerate}
\item $g$ is convex, continuously differentiable, and has $L_g$-Lipschitz
      gradient.
\item Each $\psi_i$ is convex and continuously differentiable.
\item The numerical active-set tolerances satisfy
      \(\eps_k\ge0\) and \(E_\eps:=\sum_{k=0}^\infty\eps_k<\infty\).
\item $F$ is bounded below and the level set
      \(\{x:\ F(x)\le F(x^0)+E_\eps\}\) is compact.
\item $g+\frac{\sigma}{2}\norm{\cdot-x^k}^2$ is $\mu$-strongly convex
      uniformly in $k$ for some $\mu>0$.  Equivalently, either $g$ is
      strongly convex and $\sigma\ge0$, or $\sigma>0$ supplies the missing
      strong convexity.
\end{enumerate}
\end{assumption}

The next lemma is the standard descent estimate for proximal DCA steps, with
the additional error caused by numerical active sets; see, for example, the
convergence-analysis framework in \cite{niu2022convergence}.

\begin{lemma}[Numerical-active-set descent]\label{lem:descent}
Under Assumption \ref{ass:basic}, suppose
\[
        v^k\in\conv\{\nabla\psi_i(x^k):i\in\cA_k\},
        \qquad
        \cA_k=\{i:\ h(x^k)-\psi_i(x^k)\le\eps_k\},
\]
and \(x^{k+1}\) solves \eqref{eq:update}.  Then
\begin{equation}\label{eq:descent}
        F(x^{k+1})\le
        F(x^k)+\eps_k-\frac{\mu}{2}\norm{x^{k+1}-x^k}^2 .
\end{equation}
Consequently, the iterates remain in the enlarged level set from Assumption
\ref{ass:basic}, \(\sum_k\norm{x^{k+1}-x^k}^2<\infty\), and
\(\norm{x^{k+1}-x^k}\to0\).
\end{lemma}

\begin{proof}
Let
\[
        q_k(x)=g(x)-\ip{v^k}{x-x^k}
        +\frac{\sigma}{2}\norm{x-x^k}^2 .
\]
By Assumption \ref{ass:basic}, \(q_k\) is \(\mu\)-strongly convex, and
\(x^{k+1}\) is its minimizer.  A \(\mu\)-strongly convex function satisfies
the quadratic growth inequality
\[
        q_k(y)\ge q_k(x^{k+1})
        +\frac{\mu}{2}\norm{y-x^{k+1}}^2
        \quad\hbox{for all }y .
\]
Taking \(y=x^k\) and writing \(s^k=x^{k+1}-x^k\) gives
\begin{equation*}
        q_k(x^{k+1})+\frac{\mu}{2}\norm{s^k}^2
        \le q_k(x^k)=g(x^k),
\end{equation*}
or equivalently
\begin{equation}\label{eq:qdescent}
        g(x^{k+1})-\ip{v^k}{s^k}
        +\frac{\sigma}{2}\norm{s^k}^2
        \le g(x^k)-\frac{\mu}{2}\norm{s^k}^2 .
\end{equation}
Write \(v^k=\sum_{i\in\cA_k}\alpha_i\nabla\psi_i(x^k)\), where
\(\alpha_i\ge0\) and \(\sum_{i\in\cA_k}\alpha_i=1\).  Convexity of each
\(\psi_i\) gives
\[
        h(x^{k+1})
        \ge
        \sum_{i\in\cA_k}\alpha_i\psi_i(x^{k+1})
        \ge
        \sum_{i\in\cA_k}\alpha_i\psi_i(x^k)+\ip{v^k}{s^k}.
\]
Because \(i\in\cA_k\) implies \(\psi_i(x^k)\ge h(x^k)-\eps_k\), we obtain
\[
        -h(x^{k+1})\le -h(x^k)+\eps_k-\ip{v^k}{s^k}.
\]
Adding this inequality to \eqref{eq:qdescent} and dropping the nonnegative
term \(\frac{\sigma}{2}\norm{s^k}^2\) gives \eqref{eq:descent}.

Let \(F_{\inf}\) be a lower bound of \(F\).  Summing
\eqref{eq:descent} from \(k=0\) to \(N\) yields
\[
        \frac{\mu}{2}\sum_{k=0}^N\norm{x^{k+1}-x^k}^2
        \le F(x^0)-F(x^{N+1})+\sum_{k=0}^N\eps_k
        \le F(x^0)-F_{\inf}+\sum_{k=0}^\infty\eps_k .
\]
The same recursion gives
\[
        F(x^{N+1})\le F(x^0)+\sum_{k=0}^N\eps_k\le F(x^0)+E_\eps ,
\]
so the iterates remain in the compact enlarged level set.  Letting
\(N\to\infty\) proves the summability, and therefore
\(\norm{x^{k+1}-x^k}\to0\).
\qed
\end{proof}

\subsection{Random embeddings and sampling budget}
\label{sec:embeddingbudget}

The random directions enter only through a subspace-embedding property.  At
iteration \(k\), after \(x^k\) and the active set have been formed, the only
vectors whose norms must be compared are linear combinations of the gradient
vectors already present in the safeguard.  Let
\begin{equation*}
        \mathcal{S}_k=\operatorname{span}
        \{g_k,\nabla\psi_i(x^k):i\in\cA_k\}.
\end{equation*}
Thus the embedding is required on the random but finite-dimensional subspace
\(\mathcal S_k\), conditionally on the past, not on all of \(\R^n\).

\begin{assumption}\label{ass:embed}
For each $k$, $D_k$ is independent of the past and satisfies, with
probability at least $1-\delta_k$,
\begin{equation}\label{eq:embedding}
        (1-\eta)\norm{z}\le\norm{D_kz}\le(1+\eta)\norm{z}
        \quad\hbox{for all }z\in\mathcal{S}_k ,
\end{equation}
where $\eta\in(0,1)$ and $\sum_k\delta_k<\infty$.
\end{assumption}

\begin{lemma}[Random embedding of the active span]\label{lem:randomembed}
Let \(\mathcal S\subset\R^n\) be a fixed \(d\)-dimensional subspace.  Let
\(D\in\R^{m\times n}\) have either independent Gaussian rows with covariance
\(I/m\), or independent rows sampled uniformly from the unit sphere and
scaled by \(\sqrt{n/m}\).  There is a numerical constant \(C\) such that,
for any \(\eta\in(0,1)\) and \(\delta\in(0,1)\), if
\begin{equation}\label{eq:samplecomplexity}
        m \ge C\eta^{-2}\left(d+\log\frac{1}{\delta}\right),
\end{equation}
then, with probability at least \(1-\delta\),
\[
        (1-\eta)\norm{z}\le\norm{Dz}\le(1+\eta)\norm{z}
        \quad\hbox{for all }z\in\mathcal S .
\]
\end{lemma}

\begin{proof}[Proof sketch]
Let \(U\in\R^{n\times d}\) have orthonormal columns spanning
\(\mathcal S\).  For \(z=Uy\), the desired estimate is equivalent to
all singular values of \(DU\) lying in \([1-\eta,1+\eta]\).  In the Gaussian
case, \(DU\) has the distribution of an \(m\times d\) Gaussian matrix with
entries of variance \(1/m\), so standard smallest- and largest-singular-value
concentration gives \eqref{eq:samplecomplexity}.  The scaled spherical case
is the corresponding isotropic bounded-row subspace embedding.  These are
standard oblivious subspace-embedding estimates; see
\cite{sarlos2006improved,halko2011finding,woodruff2014sketching}.
\qed
\end{proof}

Applying Lemma \ref{lem:randomembed} conditionally on the history justifies
Assumption \ref{ass:embed}: at the moment \(D_k\) is drawn, \(\mathcal S_k\)
is fixed, and the failure probability is at most \(\delta_k\).  Since
\(\sum_k\delta_k<\infty\), the embedding failures occur only finitely often
almost surely by the Borel--Cantelli lemma, which states that events with
summable probabilities occur only finitely many times almost surely.  Moreover
\(d_k:=\dim(\mathcal{S}_k)\le \min\{n,|\cA_k|+1\}\).  Thus a sufficient
per-iteration budget is
\begin{equation}\label{eq:mkbudget}
        m_k \ge C\eta^{-2}
        \left(d_k+\log\frac{1}{\delta_k}\right).
\end{equation}
The worst-case scaling is therefore linear in the ambient dimension \(n\),
not exponential in the number of possible active subsets; when the active
gradients are low rank, the required number of sampled directions is governed
by the active-span dimension \(d_k\).  For a finite horizon \(K\), choosing
\(\delta_k=\delta/K\) gives the uniform high-probability budget
\(m_k=O(\eta^{-2}(d_k+\log(K/\delta)))\).  For an infinite run, a summable
choice such as \(\delta_k=\delta/(k+1)^2\) preserves almost-sure convergence
and adds only a logarithmic dependence on \(k\).

\subsection{Almost-sure directional stationarity}
\label{sec:asstationarity}

\begin{assumption}\label{ass:numactive}
For every convergent subsequence \(x^k\to\bar x\) and every
\(i\in\cA(\bar x)\), there is a further subsequence, not relabeled, such
that \(i\in\cA_k\) for all sufficiently large \(k\).  Equivalently, no
index active at an accumulation point disappears permanently from the
numerical active sets along a sequence converging to that point.
\end{assumption}

\begin{remark}[On numerical active-set consistency]\label{rem:numactive}
Assumption \ref{ass:numactive} is the only place where the analysis uses more
than the exact active set at \(x^k\).  Its role is to rule out a numerical
active-set artifact: an index can be active at a limiting tie \(\bar x\) even
if it is inactive along a one-sided sequence converging to \(\bar x\).  Thus
using exact active sets alone does not automatically imply the assumption.
For any fixed positive active-set tolerance, however, the property holds on a
convergent tail by continuity: if \(i\in\cA(\bar x)\) and \(x^k\to\bar x\),
then \(h(x^k)-\psi_i(x^k)\to0\).  This is the finite-horizon interpretation
used in the numerical experiments.  In the asymptotic theorem, where
\(\sum_k\eps_k<\infty\), the assumption records the corresponding requirement
on the tolerance schedule: every piece active at an accumulation point must be
included infinitely often, or equivalently its active gap must be no larger
than \(\eps_k\) along a further convergent subsequence.  A slower or adaptive
active-set tolerance is the practical way to enforce this condition when
one-sided ties are expected.
\end{remark}

\begin{theorem}[Directional stationarity]\label{thm:dstationarity}
Suppose Assumptions \ref{ass:basic}, \ref{ass:embed}, and
\ref{ass:numactive} hold, and let \(\tau_k\downarrow0\).  If RA-DCA uses the
\(\eps_k\)-active sets \(\cA_k\) and the vertex safeguard
\eqref{eq:sampledres}--\eqref{eq:vertexchoice}, then every accumulation
point of $\{x^k\}$ is directional-stationary with probability one.
\end{theorem}

\begin{proof}
By Assumption \ref{ass:embed} and the preceding Borel--Cantelli argument,
the embedding event \eqref{eq:embedding} fails only finitely many times with
probability one.  Work on this event.  Lemma \ref{lem:descent} gives
$\norm{x^{k+1}-x^k}\to0$ and compactness gives accumulation points.

Let $\bar x$ be an accumulation point and take a subsequence
$x^k\to\bar x$ along which the embedding event holds.  Suppose, for
contradiction, that $R_{\rm d}(\bar x)>0$.  Then there is an index
\(\bar i\in\cA(\bar x)\) such that
\[
        \norm{\nabla\psi_{\bar i}(\bar x)-\nabla g(\bar x)}>0 .
\]
By Assumption \ref{ass:numactive}, after passing to a further subsequence we
have \(\bar i\in\cA_k\) for all sufficiently large \(k\).
By continuity, there is \(\rho>0\) such that
\begin{equation*}
        \norm{\nabla\psi_{\bar i}(x^k)-\nabla g(x^k)}\ge \rho
        \quad\hbox{for all sufficiently large }k .
\end{equation*}
The embedding lower bound gives
\[
        \widehat R_k
        \ge
        \norm{D_k(\nabla\psi_{\bar i}(x^k)-\nabla g(x^k))}
        \ge (1-\eta)\rho .
\]
Thus, for large \(k\), the safeguard is triggered because \(\tau_k\to0\).
RA-DCA therefore uses the active vertex selected in
\eqref{eq:vertexchoice},
\[
        v^k=\nabla\psi_{i_k}(x^k),
        \qquad
        \norm{D_k(v^k-\nabla g(x^k))}=\widehat R_k .
\]
The embedding upper bound also gives
\[
        \norm{v^k-\nabla g(x^k)}
        \ge
        \frac{\widehat R_k}{1+\eta}
        \ge
        \frac{1-\eta}{1+\eta}\rho .
\]

For these \(k\), \(x^{k+1}\) is the solution of \eqref{eq:update} with this
vertex.  The first-order condition is
\begin{equation*}
        \nabla g(x^{k+1})-v^k+\sigma(x^{k+1}-x^k)=0 .
\end{equation*}
Using the Lipschitz continuity of $\nabla g$,
\begin{equation*}
        \frac{1-\eta}{1+\eta}\rho
        \le \norm{v^k-\nabla g(x^k)}
        \le (L_g+\sigma)\norm{x^{k+1}-x^k}.
\end{equation*}
Thus the step length is bounded away from zero on this subsequence, which
contradicts Lemma \ref{lem:descent}.  Therefore $R_{\rm d}(\bar x)=0$, and
Lemma \ref{lem:dres} proves that $\bar x$ is directional-stationary.
\qed
\end{proof}

Theorem \ref{thm:dstationarity} is the main convergence result.  For
nonasymptotic scale, Appendix \ref{sec:residualcomplexity} records the
corresponding worst-iterate residual bound obtained by combining the
safeguard with the descent estimate.  In particular, if the safeguard
tolerance is of order \(N^{-1/2}\) or smaller and the numerical active-set
errors are summable, the best
directional-stationarity residual over the last half of the first \(N\)
iterations is also \(O(N^{-1/2})\) on the successful embedding events.  This
is a stationarity-complexity estimate, not a linear convergence theorem;
faster rates would require additional local regularity beyond the active-set
mechanism studied here.

\begin{remark}
If the vertex safeguard is removed, the same descent proof still applies,
but the limiting condition is the weaker criticality
$\nabla g(\bar x)\in\partial h(\bar x)$.  The LP \eqref{eq:lp} should
therefore be viewed as a computational stabilization and active-set
compression device, not as a standalone certificate of
directional stationarity.  The proof of Theorem \ref{thm:dstationarity}
therefore uses the LP only to generate an admissible DCA subgradient when the
sampled vertex residual is already small; the logical force of the theorem is
the vertex safeguard.
\end{remark}

\subsection{Finite block extensions}
\label{sec:blockextension}

The single finite maximum in \eqref{eq:maxh} is the basic block in a broader
class of objectives with a finite sum of max terms,
\begin{equation}\label{eq:blockmodel}
        F_{\rm b}(x)=g(x)-\sum_{\ell=1}^p h_\ell(x),
        \qquad
        h_\ell(x)=\max_{i\in\cI_\ell}\psi_{\ell i}(x).
\end{equation}
Let
\[
        \cA_\ell(x)=\argmax_{i\in\cI_\ell}\psi_{\ell i}(x)
\]
and define the set of aggregate active vertices
\begin{equation}\label{eq:blockvertices}
        \cV(x)=
        \left\{
        \sum_{\ell=1}^p \nabla\psi_{\ell i_\ell}(x):
        i_\ell\in\cA_\ell(x),\ \ell=1,\ldots,p
        \right\}.
\end{equation}
The block directional-stationarity residual is
\begin{equation}\label{eq:blockres}
        R_{\rm b}(x)=
        \max_{v\in\cV(x)}\norm{\nabla g(x)-v}.
\end{equation}
Since the directional derivative of the concave part is the support function
of \(\cV(x)\), the same separation argument as in Lemma \ref{lem:dres} gives
\[
        x \hbox{ is directional-stationary for \eqref{eq:blockmodel}}
        \quad\Longleftrightarrow\quad
        R_{\rm b}(x)=0 .
\]
The corresponding convex criticality residual satisfies
\[
        \dist\left(\nabla g(x),\partial\sum_{\ell=1}^p h_\ell(x)\right)
        \le R_{\rm b}(x),
\]
and the inequality may again be strict.

For a block version of the method, let
\[
        \cA_{\ell k}
        =
        \{i:\ h_\ell(x^k)-\psi_{\ell i}(x^k)\le\eps_{\ell k}\}
\]
and let \(\cV_k^\eps\) be the aggregate vertices formed from these numerical
block active sets:
\[
        \cV_k^\eps
        =
        \left\{
        \sum_{\ell=1}^p \nabla\psi_{\ell i_\ell}(x^k):
        i_\ell\in\cA_{\ell k},\ \ell=1,\ldots,p
        \right\}.
\]
The embedding subspace is
\[
        \mathcal S_k^{\rm b}
        =
        \operatorname{span}\{\nabla g(x^k),v:\ v\in\cV_k^\eps\},
\]
and the sampled residual is
\begin{equation}\label{eq:blocksampledres}
        \widehat R^{\rm b}_k
        =
        \max_{v\in\cV_k^\eps}
        \norm{D_k(v-\nabla g(x^k))}.
\end{equation}

\begin{corollary}[Block extension]\label{cor:blockextension}
Consider \eqref{eq:blockmodel} and suppose the analogues of Assumptions
\ref{ass:basic} and \ref{ass:numactive} hold for the numerical block active
sets.  Suppose the block tolerances satisfy
\[
        \sum_{k=0}^\infty\sum_{\ell=1}^p\eps_{\ell k}<\infty .
\]
Assume that \(D_k\) satisfies \eqref{eq:embedding} on
\(\mathcal S_k^{\rm b}\), with summable failure probabilities.  If the block
method selects an aggregate vertex attaining \(\widehat R_k^{\rm b}\) whenever
\(\widehat R_k^{\rm b}>\tau_k\), with \(\tau_k\downarrow0\), then every
accumulation point is directional-stationary for \eqref{eq:blockmodel} with
probability one.
\end{corollary}

\begin{proof}
If \(v^k\in\operatorname{conv}\cV_k^\eps\), the proof of Lemma
\ref{lem:descent} applies block by block and gives the same descent estimate
with \(\eps_k\) replaced by \(\sum_{\ell=1}^p\eps_{\ell k}\).  This includes
both an aggregate vertex and a convex combination of aggregate vertices from
\(\cV_k^\eps\).  Work on the
almost-sure event on which the embedding on \(\mathcal S_k^{\rm b}\) fails
only finitely many times.  If an accumulation point \(\bar x\) had
\(R_{\rm b}(\bar x)>0\), numerical active-set consistency would give an
aggregate vertex \(v^k\in\cV_k^\eps\) converging to a limiting active
aggregate vertex whose distance from \(\nabla g(\bar x)\) is positive.  Hence
its distance from \(\nabla g(x^k)\) stays bounded away from zero along a
subsequence.  The lower embedding bound would then keep
\(\widehat R_k^{\rm b}\) bounded away from zero, so the safeguard
would select such an aggregate vertex for all large \(k\).  The first-order
condition for the DCA subproblem would imply, as in Theorem
\ref{thm:dstationarity}, that \(\norm{x^{k+1}-x^k}\) is bounded away from zero
on this subsequence, contradicting Lemma \ref{lem:descent}.  Thus
\(R_{\rm b}(\bar x)=0\), which is equivalent to directional stationarity for
\eqref{eq:blockmodel}.
\qed
\end{proof}

The price of the block statement is that \(\cV_k^\eps\) may be a Cartesian
product of numerical block active sets.  Thus Corollary
\ref{cor:blockextension} is most useful as a stationarity benchmark: a
practical block implementation should be viewed according to how closely its
active-vertex search approximates the aggregate safeguard
\eqref{eq:blocksampledres}.
Appendix \ref{sec:appendixblock} gives a fixed-factor approximate version of
this statement and records the corresponding inexact subproblem conditions.
Greedy block rules without a verified approximation
factor should be interpreted computationally: they test whether the same
sketched active-set geometry is useful before the exact numerical aggregate
safeguard is affordable.

The top-\(k\), trimmed-regression, complementarity/binary, and QUBO
experiments in Sections \ref{sec:topksupport}, \ref{sec:trimmedreg},
\ref{sec:binarycomp}, and \ref{sec:qubo} use this block
viewpoint.  Their active-vertex search is implemented by a greedy full-space
or sketched rule rather than by enumerating all aggregate vertices, so those
results are reported as computational block extensions of the single-max
method.

\begin{remark}[Full-vertex scans in single and block models]
For an explicit single maximum, the full-vertex residual scan is simply
\(\max_{i\in\cA(x)}\|\nabla\psi_i(x)-\nabla g(x)\|\).  If all active
gradients are already available, this costs \(O(n|\cA(x)|)\) and is often the
most direct implementation; choosing \(D=I\) in \eqref{eq:sampledres} gives
the same full-vertex scan.  The situation changes for \eqref{eq:blockmodel}.  An
active vertex is now an aggregate
\(\sum_{\ell=1}^p\nabla\psi_{\ell i_\ell}(x)\), and exact full-vertex scanning
requires a search over
\(\cA_1(x)\times\cdots\times\cA_p(x)\), whose size is
\(\prod_{\ell=1}^p|\cA_\ell(x)|\).  This combinatorial growth is the main
difficulty for exact scans.

The block experiments therefore use a different object from the exact
safeguard in Corollary \ref{cor:blockextension}: a greedy aggregate search.
The full-space greedy reference evaluates greedy choices in the original
\(\R^n\) space; the RA block rule performs the same type of greedy search
after projecting each block gradient by \(D\).  Thus its active-set cost is
tied to forming and scanning the projected block gradients, roughly
\(O(m\sum_\ell|\cA_\ell(x)|)\) once the block gradients are available, rather
than enumerating \(\prod_\ell|\cA_\ell(x)|\) aggregate vertices.  This is why
the block rule is reported as an approximate computational extension: it can
track a full-space greedy aggregate choice, but it is not the exact aggregate
full-vertex scan unless an additional approximation guarantee such as
Appendix \ref{sec:appendixblock} is verified.
\end{remark}

\section{Implementation}
\label{sec:implementation}

The MATLAB implementation follows Algorithm \ref{alg:radca}.  For a given
active gradient matrix $G_k$ and gradient $g_k$, the products
\begin{equation*}
        A_k=D_kG_k,\qquad b_k=D_kg_k
\end{equation*}
are computed by one dense matrix multiply and one matrix-vector
multiply.  If a compatible GPU is available, the code optionally creates
$D_k$, $G_k$, and $g_k$ as \texttt{gpuArray} objects and gathers only the
small projected matrices needed by the LP solver.  The default LP backend is
\texttt{lpSolver="gurobi"}, which calls Gurobi as a continuous LP solver.  If
Gurobi is unavailable, the code falls back to \texttt{linprog} and then to a
projected-gradient solve of the sampled least-squares projection over the
simplex.  Appendix \ref{sec:lpbackendappendix} reports diagnostics for this
LP layer: Gurobi is the fastest exact backend on the test machine, while the
projected simplex method is a portable approximate fallback.  The LP is
skipped on iterations where the sampled vertex residual exceeds the safeguard
tolerance, because the active vertex is then chosen directly.
The experiment tables report CPU wall-clock times, with GPU
acceleration disabled unless a table explicitly says otherwise.  The separate
QUBO backend diagnostic in Appendix \ref{sec:quboappendix} reports a genuinely
batched GPU prototype alongside serial and batched CPU implementations.

The sketch and LP are only the active-subgradient selection layer.  Once
\(v^k\) has been chosen, the dominant cost in a full implementation is often
the convex DCA subproblem \eqref{eq:update}.  RA-DCA should therefore be paired
with a solver that exploits the structure of this subproblem: closed-form
updates when available, projected first-order methods for simple bound
constraints, sparse linear algebra for quadratic models, or an interior-point,
conic, or commercial convex-optimization solver when the model calls for it.
GPU acceleration is useful in exactly this structural sense: if the subproblem
solver is dominated by large matrix-vector products, batched projections, or
other operations with enough arithmetic intensity to amortize data transfer,
then the whole RA-DCA iteration can benefit from running those operations on a
GPU.  If the subproblem is small, sparse, or dominated by branching logic or
factorizations better handled by a CPU solver, a specialized CPU implementation
may be preferable for that case.

\section{Numerical experiments}
\label{sec:numerics}

We report reproducible MATLAB experiments for several
max-structured DC instances.  The experiments separate three regimes:
single finite maxima covered by Theorem \ref{thm:dstationarity}, block
sum-of-max models with combinatorial aggregate active sets, and harder
penalty models where starts, the DC split, or a reference mixed-integer solver
also affect the outcome.  The sequence mirrors the model classes in Section
\ref{sec:intromodels}; the QUBO benchmark is kept last because it adds a
nonconvex quadratic term, two DC decompositions, multistart initialization,
and a Gurobi MIQP reference.  The tables report either averages over the
stated random instances or sketches, or instancewise diagnostics when that
view is more informative.  Table \ref{tab:testoverview} summarizes the test
suite.

The tolerance sequences in Algorithm \ref{alg:radca} should be interpreted
asymptotically in the convergence theorem.  The reported experiments are
finite-horizon computations and use fixed small tolerances instead:
\(\eps_k=\tau_k=10^{-10}\) for the synthetic finite-max tests,
\(\eps_k=\tau_k=10^{-12}\) for the sparse support-function tests, and
tie tolerances of \(10^{-8}\) in the block penalty implementations.  These
fixed values do not themselves satisfy the infinite-horizon assumptions
\(\sum_k\eps_k<\infty\) and \(\tau_k\downarrow0\); they are finite-precision
approximations of the asymptotic schedules described after Algorithm
\ref{alg:radca}.

\begin{table}[!ht]
\centering
\caption{Overview of the numerical experiments.}
\label{tab:testoverview}
\scriptsize
\begingroup
\setlength{\tabcolsep}{3pt}
\renewcommand{\arraystretch}{1.08}
\begin{tabular}{@{}L{0.16\linewidth}L{0.27\linewidth}L{0.24\linewidth}L{0.25\linewidth}@{}}
\toprule
test & model or data & subproblem and initialization & role in the comparison\\
\midrule
Signed-pair max-affine &
Synthetic finite maxima with \(n=50,100,200,500\) and \(p=5n\) signed pairs. &
Closed-form DCA step; all runs start at \(x^0=0\). &
Tests whether sampled active-set geometry finds stronger vertices.\\
\addlinespace
Sketch-size study &
Same signed-pair max-affine model with \(n=100\), \(p=500\). &
One-step active-vertex screening from the degenerate start. &
Measures how the number \(m\) of sampled directions affects active-vertex
quality and projection cost.\\
\addlinespace
Max-quadratic pieces &
Signed-pair smooth pieces
\(\psi_i(x)=a_i^\top x+\gamma\|x\|^2/2\), \(\gamma=0.25\). &
Closed-form quadratic DCA step; start \(x^0=0\). &
Checks whether the same effect persists when active gradients change with
the iterate.\\
\addlinespace
LIBSVM support function &
Signed support-function models from sparse LIBSVM data sets. &
One-step max-affine DCA from \(x^0=0\). &
Tests screening on large sparse active-gradient matrices.\\
\addlinespace
Top-\(k\) support function &
Top-\(k\) signed support models from the same sparse LIBSVM matrices. &
One-step aggregate-vertex DCA from \(x^0=0\), using full-space and sketched
greedy selection. &
Tests a trimmed/sparse block active set where exact aggregate enumeration is
combinatorial.\\
\addlinespace
Trimmed sparse regression &
Trimmed ridge least-squares models on sparse LIBSVM matrices. &
One-step aggregate-vertex DCA from \(w^0=0\); the ridge subproblem is solved
by sparse linear algebra. &
Tests a robust-estimation top-residual model and exposes the limits of
stationarity-oriented aggregate heuristics.\\
\addlinespace
LCP/MPEC penalty &
Sparse synthetic row-stochastic LCP maps with \(n=100,300,500\). &
Closed-form box projection for a sum of two-piece max blocks. &
Tests greedy aggregate safeguards on nonsymmetric complementarity and
binary-penalty blocks.\\
\addlinespace
OR-Library QUBO &
UBQP instances \(bqp50\), \(bqp100\), and \(bqp250\), converted to minimization
QUBO form. &
Box QP solved by projected gradient with active-set polish and multistart. &
Larger benchmark with shift/spectral DC splits and a Gurobi MIQP reference.\\
\bottomrule
\end{tabular}
\endgroup
\end{table}

Unless stated otherwise, the column ``CPU s'' reports average wall-clock
seconds per run, measured by \texttt{tic}/\texttt{toc} around the solver call,
on a Windows workstation with a 12th Gen Intel Core i9-12900K processor.  GPU
acceleration was disabled for the main CPU timing tables.  The QUBO
\(bqp250.8\) backend diagnostic is the exception: it separately reports
serial CPU, batched CPU, and batched GPU timings on the same machine.  The
experiments use MATLAB R2023b syntax and require only base MATLAB for the
fallback mode, excluding the Gurobi reference solves and the optional GPU
diagnostics.  The reported RA-DCA runs set \texttt{lpSolver="gurobi"}
explicitly when the active-set LP backend is used.

\subsection{Compared active-gradient selection rules}

The single-max experiments compare the four active-gradient selection rules
defined in Section \ref{sec:algorithm}: centered DCA, random-vertex DCA,
full-vertex DCA, and RA-DCA.  Once a rule has chosen \(v^k\), all four methods
solve the same DCA subproblem \eqref{eq:update}.  Thus, ``centered DCA''
should not be read as the whole class of traditional DCA methods.  It is one
particular traditional DCA baseline with an averaged active-subgradient
selection rule.  The comparison is summarized in Table \ref{tab:methods}.

\begin{table}[!ht]
\centering
\caption{Active-gradient selection rules used in the numerical experiments.}
\label{tab:methods}
\small
\begin{tabular}{lp{0.68\linewidth}}
\toprule
method & active-gradient selection and role\\
\midrule
centered DCA &
$v^k=|\cA_k|^{-1}\sum_{i\in\cA_k}\nabla\psi_i(x^k)$; a traditional DCA
baseline using the averaged active subgradient.\\
random-vertex DCA &
uniformly samples one $\nabla\psi_i(x^k)$ with $i\in\cA_k$; a single
active-piece baseline.\\
full-vertex DCA &
chooses an active gradient maximizing
\(\norm{\nabla\psi_i(x^k)-\nabla g(x^k)}\) over \(i\in\cA_k\); an ideal
full-active-set screening baseline.\\
RA-DCA &
checks the sampled vertex residual first; if the safeguard triggers, it uses
the most violating active vertex, and otherwise solves the LP for an active
convex combination.\\
\bottomrule
\end{tabular}
\end{table}

The first baseline is useful because it can stop at critical points that are
not directionally stationary.  The second baseline tests whether simply
choosing an active vertex is enough; its
performance depends on which active piece is selected.  The full-vertex
baseline shows what is obtained when the exact vertex residual can be scanned
without sketching.  RA-DCA uses sampled active-set geometry to choose between
a stable convex combination step and a safeguarded active vertex.  On explicit
single-max instances, the full-vertex scan is a strong benchmark and may be
cheaper; RA-DCA is designed to approximate the same active-set decision
when full gradients, full-space scans, or aggregate vertices are not the
natural computational primitive.
The full-vertex baseline is included in the single-max tests, where the scan
is over the ordinary active set \(\cA_k\).  In the block sum-of-max
experiments, the corresponding exact baseline would have to scan aggregate
vertices in the Cartesian product of the block active sets.  For binary
penalties this can mean \(2^n\) sign patterns at a tied point, so we do not
use it as a routine baseline.  Instead, the same active-set idea is applied
blockwise and reported as ``RA-DCA block'' to distinguish it from the
single-max theorem and the exact aggregate-vertex safeguard in Corollary
\ref{cor:blockextension}.

The active convex-combination LP \eqref{eq:lp} is not solved on every RA-DCA
iteration.  The implementation first evaluates the sampled vertex residual;
only when that residual is below the safeguard tolerance does it solve the LP.
This ordering matters because the LP is an active-gradient centering device,
whereas the safeguard is the directional-stationarity test.  In the main
main finite-max tests, the sampled vertex residual is usually large until the
method reaches a single active piece, so the LP branch is rarely on the
critical path.  Thus the finite-max tables should not be read as evidence that
the LP itself improves the objective value; they test the vertex-first
safeguard.  Appendix \ref{sec:lpbackendappendix} verifies that the LP fallback
is inexpensive when it is forced or needed, and gives a near-active diagnostic
where the LP prevents a spurious step caused by an \(\eps_k\)-active set.

\subsection{Subproblems and initialization}

For the single finite-max tests in Sections \ref{sec:maxaffine}--\ref{sec:libsvm},
\(g(x)=\frac12\norm{x}^2\).  Hence,
after an active subgradient \(v^k\in\partial h(x^k)\) has been selected, the
DCA subproblem is
\begin{equation*}
        x^{k+1}=\argmin_x\left\{\frac12\norm{x}^2-(v^k)^\top x\right\}=v^k.
\end{equation*}
Thus the reported differences in the affine and max-quadratic tests are
entirely due to active-subgradient selection, not to numerical error in the
convex subproblem solve.  The signed-pair max-affine and max-quadratic tests
all start from \(x^0=0\), where the signed pairs are simultaneously active.

For the binary/complementarity and QUBO block tests, the feasible set is the
box \([0,1]^n\).  The binary-complementarity subproblem has a closed-form box
projection, while the
QUBO relaxation gives a bound-constrained convex quadratic program.  Because
the only constraints are simple bounds, it is solved by a local accelerated
projected-gradient routine with componentwise clipping, followed by a small
active-set polish.  Initial points are stated separately in the corresponding
subsections because they materially affect which binary-penalty blocks are
active.

\subsection{Degenerate max-affine instances}
\label{sec:maxaffine}

We also test
\begin{equation}\label{eq:maxaffineexp}
        F(x)=\frac12\norm{x}^2-\max_{1\le i\le 2p} a_i^\top x,
\end{equation}
where the rows occur in signed pairs
$a_{p+i}=-a_i$.  Thus \(p\) is the number of signed pairs, and the maximum
contains \(2p\) affine pieces.  Starting from $x^0=0$, all pieces are active.
A centered active-subgradient DCA step again stalls, while RA-DCA uses the
projected vertex residual to select a large active gradient.

\begin{table}[!ht]
\centering
\caption{Results for signed-pair max-affine instances.  The reported
values are means over ten random instances for each dimension.  Here \(p\)
denotes the number of signed pairs, so the maximum has \(2p\) affine pieces.
CPU s is the mean CPU wall-clock time per run.}
\label{tab:maxaffine}
\scriptsize
\begin{tabular}{rrlrrrr}
\toprule
$n$ & $p$ & method & objective & $R_{\rm d}$ & iterations & CPU s\\
\midrule
50 & 250 & centered DCA & -0.0000 & 1.99e+00 & 20.0 & 0.00785\\
50 & 250 & random-vertex DCA & -1.1887 & 0.00e+00 & 2.3 & 0.00117\\
50 & 250 & full-vertex DCA & -1.9750 & 0.00e+00 & 2.0 & 0.00137\\
50 & 250 & RA-DCA & -1.8246 & 0.00e+00 & 2.0 & 0.00347\\
100 & 500 & centered DCA & -0.0000 & 2.00e+00 & 20.0 & 0.0115\\
100 & 500 & random-vertex DCA & -0.8226 & 0.00e+00 & 2.0 & 0.00102\\
100 & 500 & full-vertex DCA & -1.9946 & 0.00e+00 & 2.0 & 0.000881\\
100 & 500 & RA-DCA & -1.8635 & 0.00e+00 & 2.0 & 0.00175\\
200 & 1000 & centered DCA & -0.0000 & 2.00e+00 & 20.0 & 0.0855\\
200 & 1000 & random-vertex DCA & -0.9551 & 0.00e+00 & 2.0 & 0.00489\\
200 & 1000 & full-vertex DCA & -1.9983 & 0.00e+00 & 2.0 & 0.00555\\
200 & 1000 & RA-DCA & -1.8924 & 0.00e+00 & 2.0 & 0.00777\\
500 & 2500 & centered DCA & -0.0000 & 2.00e+00 & 20.0 & 0.467\\
500 & 2500 & random-vertex DCA & -0.7640 & 0.00e+00 & 2.0 & 0.0246\\
500 & 2500 & full-vertex DCA & -1.9987 & 0.00e+00 & 2.0 & 0.0274\\
500 & 2500 & RA-DCA & -1.9383 & 0.00e+00 & 2.0 & 0.0288\\
\bottomrule
\end{tabular}

\end{table}

Table \ref{tab:maxaffine} shows the targeted behaviour.  The centered DCA
baseline certifies only convex-combination criticality at the origin and
retains a residual close to \(2\).  The random-vertex baseline leaves the
origin and reaches \(R_{\rm d}=0\), but its objective values are larger because
the selected active piece is uniformly random.  The full-vertex baseline gives
the best objectives by scanning the complete active set and selecting the
largest exact vertex residual.  RA-DCA also reaches directional stationarity
in about two iterations and consistently lies between random-vertex DCA and
this full-active-set baseline, showing that the sketch selects substantially
stronger vertices than random sampling.  Because the safeguard is triggered
at the degenerate start, RA-DCA skips the LP on the expensive first
iteration; the largest tested max-affine case takes about \(27\) milliseconds
per run.
Figure \ref{fig:convergence} shows the same behaviour on one representative
instance: the centered baseline remains flat at the initial critical point,
whereas the vertex-based methods move immediately.

\begin{figure}[!ht]
\centering
\includegraphics[width=.78\linewidth]{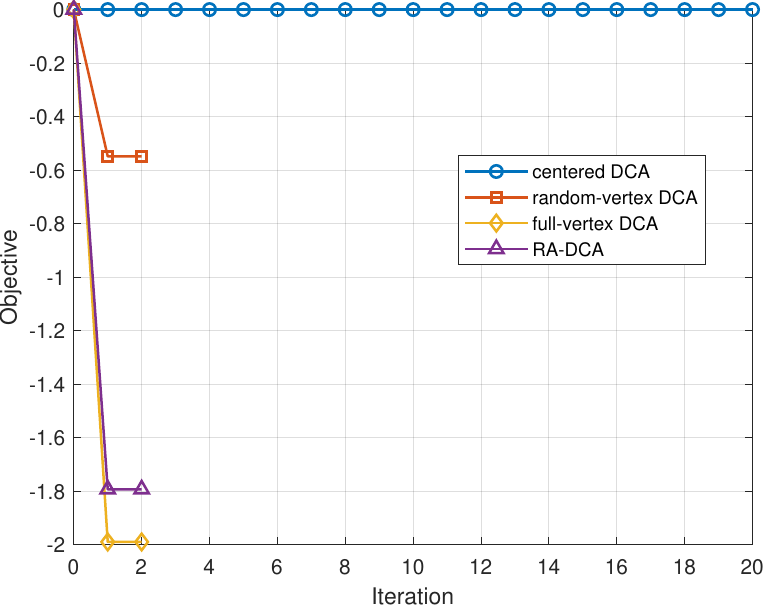}
\caption{Objective histories on one signed-pair max-affine instance with
$n=100$ and $p=500$.  Centered DCA remains at the averaged critical point,
whereas the vertex-based methods leave the origin; full-vertex DCA uses the
exact active-set scan, while RA-DCA uses the sampled residual.}
\label{fig:convergence}
\end{figure}

\subsection{Effect of the number of sampled directions}

For the signed-pair max-affine instances, the first RA-DCA step from the
origin selects an active gradient by comparing sketched norms
$\norm{D_ka_i}$.  The ideal active vertex is the one with largest true norm.
Table \ref{tab:sketch} reports, for $n=100$ and $p=500$, the ratio between
the true norm of the selected vertex and the best true norm, together with
the fraction of trials in which this ratio is at least $0.95$.  Even a small
number of directions gives a useful active-set summary.  The projection time
is the average CPU time to form \(D_kG_k\).

\begin{table}[!ht]
\centering
\caption{Effect of the number of sampled directions on active-vertex
selection for signed-pair max-affine instances with $n=100$ and $p=500$.}
\label{tab:sketch}
\small
\begin{tabular}{rrrr}
\toprule
$m$ & mean ratio & success $\ge 0.95$ & projection ms\\
\midrule
5 & 0.952 & 0.52 & 0.079\\
10 & 0.955 & 0.58 & 0.122\\
20 & 0.960 & 0.72 & 0.101\\
40 & 0.967 & 0.78 & 0.178\\
80 & 0.981 & 0.88 & 0.263\\
160 & 0.977 & 0.94 & 1.057\\
\bottomrule
\end{tabular}

\end{table}

\begin{figure}[t]
\centering
\includegraphics[width=.78\linewidth]{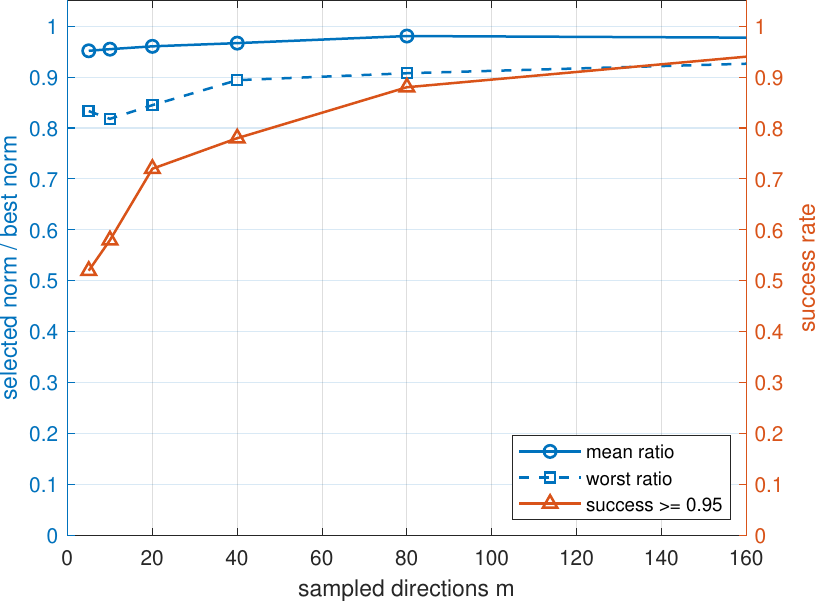}
\caption{Sampling effect on the selected active vertex.  The ratio compares
the true norm of the sketched winner with the best true active-gradient norm.}
\label{fig:sketch}
\end{figure}

The results in Table \ref{tab:sketch} and Figure \ref{fig:sketch} quantify the
computational role of the sketch.  With only \(m=5\) sampled directions, the
chosen vertex already has about \(95\%\) of the best active-gradient norm on
average, but the least favorable trial can be substantially smaller and only about half
of the trials exceed the \(0.95\) threshold.  Increasing \(m\) improves both
the average ratio and the reliability of the selected vertex, reaching a
\(0.94\) success rate at \(m=160\).  The projection remains inexpensive for
these instance sizes, supporting the use of \(D_kG_k\) as the dominant
active-set screening operation.

\subsection{Max-quadratic pieces}
\label{sec:maxquad}

The smooth-piece test uses convex components
\begin{equation*}
        \psi_i(x)=a_i^\top x+\frac{\gamma}{2}\norm{x}^2,\qquad
        \gamma=0.25,
\end{equation*}
with the same signed-pair construction.  The full objective is
\begin{equation*}
        F(x)=\frac12\norm{x}^2-\max_i
        \left(a_i^\top x+\frac{\gamma}{2}\norm{x}^2\right).
\end{equation*}
This keeps the subproblem simple: after an active piece is selected, the DCA
update is \(x^{k+1}=a_i+\gamma x^k\), or the corresponding active convex
combination.  The active gradients now change with the iterate, unlike the
max-affine test.  All runs start from \(x^0=0\).

\begin{table}[!ht]
\centering
\caption{Results for signed-pair max-quadratic instances with $\gamma=0.25$.
Values are means over ten random instances for each dimension.  CPU s reports
per-run time.}
\label{tab:maxquadratic}
\scriptsize
\begin{tabular}{rrlrrrr}
\toprule
$n$ & $p$ & method & objective & $R_{\rm d}$ & iterations & CPU s\\
\midrule
50 & 250 & centered DCA & -0.0000 & 1.99e+00 & 60.0 & 0.0251\\
50 & 250 & random-vertex DCA & -1.4438 & 1.47e-11 & 18.3 & 0.00357\\
50 & 250 & full-vertex DCA & -2.6333 & 7.23e-12 & 19.0 & 0.00413\\
50 & 250 & RA-DCA & -2.4555 & 6.98e-12 & 19.0 & 0.0044\\
100 & 500 & centered DCA & -0.0000 & 2.00e+00 & 60.0 & 0.0386\\
100 & 500 & random-vertex DCA & -1.0063 & 1.72e-11 & 18.0 & 0.00348\\
100 & 500 & full-vertex DCA & -2.6594 & 7.27e-12 & 19.0 & 0.00416\\
100 & 500 & RA-DCA & -2.5012 & 7.04e-12 & 19.0 & 0.0044\\
200 & 1000 & centered DCA & -0.0000 & 2.00e+00 & 60.0 & 0.252\\
200 & 1000 & random-vertex DCA & -1.0372 & 1.25e-11 & 18.2 & 0.00781\\
200 & 1000 & full-vertex DCA & -2.6644 & 7.27e-12 & 19.0 & 0.00973\\
200 & 1000 & RA-DCA & -2.5752 & 7.15e-12 & 19.0 & 0.0115\\
500 & 2500 & centered DCA & -0.0000 & 2.00e+00 & 60.0 & 1.41\\
500 & 2500 & random-vertex DCA & -1.0891 & 1.37e-11 & 18.2 & 0.0407\\
500 & 2500 & full-vertex DCA & -2.6649 & 7.27e-12 & 19.0 & 0.0446\\
500 & 2500 & RA-DCA & -2.6241 & 7.22e-12 & 19.0 & 0.0453\\
\bottomrule
\end{tabular}

\end{table}

\begin{figure}[t]
\centering
\includegraphics[width=\linewidth]{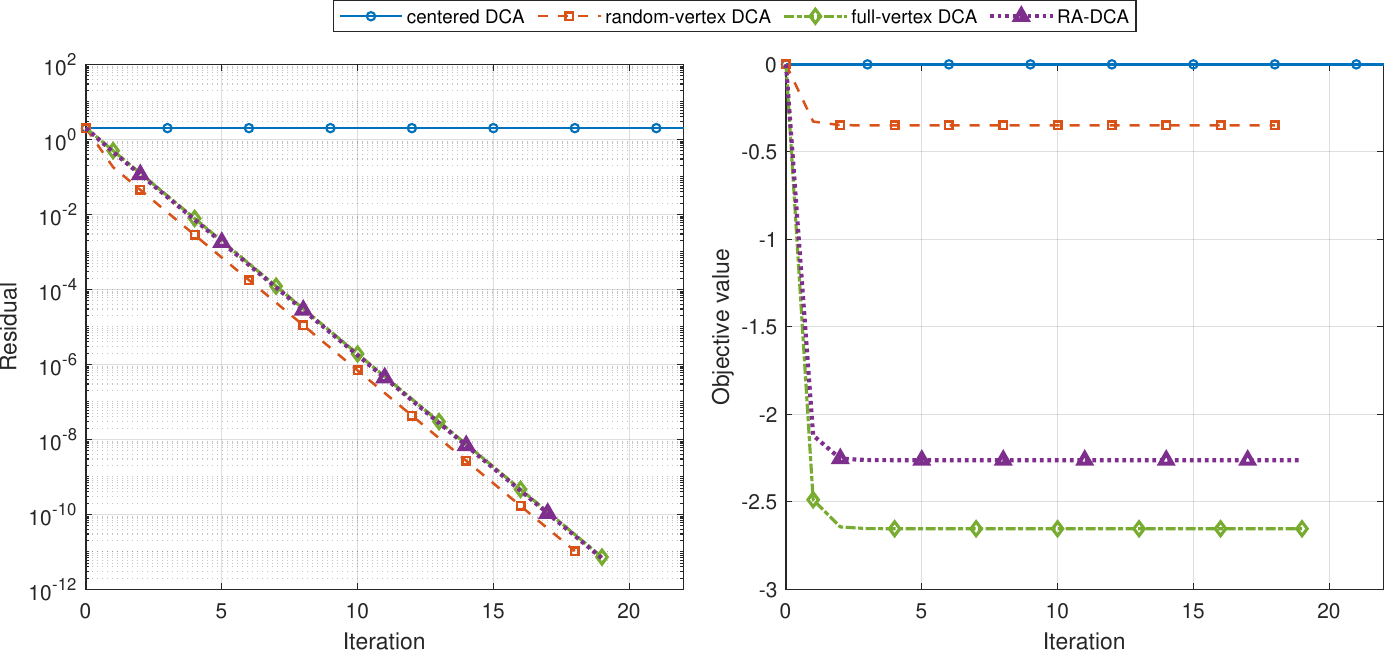}
\caption{Residual and objective histories on one max-quadratic instance with
$n=100$ and $p=500$.  Random-vertex DCA reduces the residual fastest on this
instance, but it converges to a larger objective value; RA-DCA and full-vertex
DCA reach lower stationary values.}
\label{fig:maxquadratic}
\end{figure}

Table \ref{tab:maxquadratic} shows that the same active-set effect persists
when the pieces are smooth nonlinear convex functions.  Centered DCA again
stalls at the origin with residual nearly \(2\).  The vertex-based methods
all converge to residuals on the order of \(10^{-11}\), but the objective
values separate the stationary points.  Random-vertex DCA reaches stationarity
quickly, yet its mean objective is much larger; for example, at \(n=100\) it
ends at \(-1.0063\), compared with \(-2.5012\) for RA-DCA and \(-2.6594\) for
the full-vertex reference.  The separation remains at \(n=500\): RA-DCA
reaches \(-2.6241\), close to the full-vertex value \(-2.6649\), while the
random-vertex mean is only \(-1.0891\).  Thus residual decrease alone is not a
reliable measure of active-set quality.  Figure \ref{fig:maxquadratic} shows this on
one representative instance: the random-vertex residual decreases slightly
faster, but the objective panel shows that it settles at a substantially higher
stationary value.  RA-DCA follows the full-vertex objective much more closely,
which is the targeted effect of the sampled active-set screen.

The CPU times show that the vertex-first safeguard is important in practice:
RA-DCA is now close to the full-vertex scan on these explicit synthetic active
sets, while still using the same sampled residual rule that applies when a
full scan is less attractive.

\subsection{LIBSVM support-function benchmark}
\label{sec:libsvm}

The preceding synthetic tests make the active-set geometry transparent.  To
check the same mechanism on larger sparse active-gradient matrices, we use
three binary classification data sets from the LIBSVM collection
\cite{chang2011libsvm}.  These same matrices are also used in the top-\(k\)
support and trimmed-regression tests in Sections \ref{sec:topksupport} and
\ref{sec:trimmedreg}.  Throughout these three tests,
\(N\) denotes the number of samples and \(n\) denotes the number of features,
which is also the number of optimization variables in \(w\in\R^n\).  After
removing zero rows, the data sizes are
\[
        \begin{array}{c|ccc}
        \hbox{data set} & \texttt{a8a} & \texttt{phishing} & \texttt{w8a}\\
        \hline
        N & 22696 & 11055 & 45546\\
        n & 123 & 68 & 300
        \end{array}
\]
If \(a_i^\top\) denotes one sparse data row, we form the signed
support-function model
\begin{equation}\label{eq:libsvmsupport}
        F(w)=\frac12\norm{w}^2
        -\max_{1\le i\le N}\{a_i^\top w,\,-a_i^\top w\}.
\end{equation}
This is still an exact instance of \eqref{eq:maxh}.  Starting from \(w^0=0\),
all \(2N\) signed pieces are active, and the ideal full-vertex rule selects a
data row with largest Euclidean norm.  The RA-DCA runs use the finite-horizon
direction budget \eqref{eq:mkbudget} with \(d=n\), \(C=1\), \(\eta=0.8\), and
\(\delta=0.05\).  The column ``norm ratio'' reports
\(\norm{w^1}/\max_i\norm{a_i}\), and the hit rate is the fraction of repeated
runs reaching the full-vertex objective.
This benchmark measures selection quality rather than speed against the
full-vertex rule.  In the support-function model, the exact full-vertex score
at \(w^0=0\) reduces to a closed-form sparse row-norm scan.  The row labeled
``full-vertex oracle'' reports this shortcut, not a generic active-set
screening algorithm.  RA-DCA, by contrast, keeps the same generic sampled
vertex residual used in the other finite-max experiments, with all active
signed rows in the candidate set; since the residual is large at the origin,
the LP is skipped and the most violating sketched vertex is used directly.
The test checks whether the randomized screen can recover a near-best active
row from a large active set.

\begin{table}[!ht]
\centering
\caption{LIBSVM support-function benchmark \eqref{eq:libsvmsupport}.  The
columns ``samples'' and ``features'' report \(N\) and the variable dimension
\(n\), respectively.  The full-vertex oracle uses the closed-form row-norm
scan available only for this support-function model.  The random-vertex row
averages ten random choices; the RA-DCA row uses one embedding-budget sketch.
CPU s excludes data loading.}
\label{tab:libsvmsupport}
\scriptsize
\resizebox{\linewidth}{!}{\begin{tabular}{lrrlrrrrr}
\toprule
data & samples & features & method & dirs & objective & norm ratio & hit rate & CPU s\\
\midrule
a8a & 22696 & 123 & centered DCA & -- & 0 & 0.000 & 0.00 & 0\\
a8a & 22696 & 123 & random-vertex DCA & -- & -6.9 & 0.993 & 0.90 & 0.000793\\
a8a & 22696 & 123 & full-vertex oracle & -- & -7 & 1.000 & 1.00 & 0.00259\\
a8a & 22696 & 123 & RA-DCA & 198 & -7 & 1.000 & 1.00 & 0.0714\\
\midrule
phishing & 11055 & 68 & centered DCA & -- & 0 & 0.000 & 0.00 & 0\\
phishing & 11055 & 68 & random-vertex DCA & -- & -0.5 & 1.000 & 1.00 & 0.000196\\
phishing & 11055 & 68 & full-vertex oracle & -- & -0.5 & 1.000 & 1.00 & 0.00209\\
phishing & 11055 & 68 & RA-DCA & 113 & -0.5 & 1.000 & 1.00 & 0.0276\\
\midrule
w8a & 45546 & 300 & centered DCA & -- & 0 & 0.000 & 0.00 & 0\\
w8a & 45546 & 300 & random-vertex DCA & -- & -4.25 & 0.248 & 0.00 & 0.000659\\
w8a & 45546 & 300 & full-vertex oracle & -- & -57 & 1.000 & 1.00 & 0.00456\\
w8a & 45546 & 300 & RA-DCA & 475 & -56.5 & 0.996 & 0.00 & 0.229\\
\bottomrule
\end{tabular}

}
\end{table}

Table \ref{tab:libsvmsupport} gives a more computational stress-oriented view
of the active-set rule.  The centered baseline again accepts the origin even
though the vertex residual is nonzero.  On \(a8a\) and \texttt{phishing}, row
norms are nearly uniform, so random vertices already perform well.  The
\(w8a\) instance is more discriminating: random vertices achieve only about
one quarter of the best active-gradient norm on average, whereas RA-DCA
selects a vertex with norm ratio \(0.996\).  This improvement is not free:
the sketched products use \(475\) sampled directions and all \(91{,}092\)
active signed rows, so RA-DCA remains slower than the closed-form oracle
scan.  Solving the convex-combination LP before applying the safeguard would
be wasteful in this case; the vertex-first implementation skips that LP and
takes about \(0.23\) seconds on \(w8a\).  The oracle scan is much cheaper
because this pure support-function example
exposes the exact vertex scores as row norms; if this structure is available,
that specialized scan is the appropriate implementation.  Its role here is as
an upper-quality reference.  In less special max-structured models, where
active gradients change with \(x\) or the exact screening score is expensive,
the same table indicates the quality that the randomized screen can recover
from a large active set.

\subsection{Top-\(k\) support-function benchmark}
\label{sec:topksupport}

The preceding support-function test has a special closed-form full-vertex
scan.  A top-\(k\) support model gives a sparse/trimmed block active set:
\begin{equation}\label{eq:topksupport}
        F_k(w)=\frac12\norm{w}^2
        -\sum_{j=1}^k |(Xw)|_{(j)},
\end{equation}
where \(|(Xw)|_{(1)}\ge\cdots\ge |(Xw)|_{(N)}\) are the ordered absolute
responses of the sparse LIBSVM matrix \(X\).  Equivalently,
\[
        \sum_{j=1}^k |(Xw)|_{(j)}
        =
        \max_{\substack{S\subseteq\{1,\ldots,N\}\\ |S|=k}}
        \max_{s_i\in\{-1,1\}}
        \sum_{i\in S}s_i a_i^\top w .
\]
Thus \(w^0=0\) has a combinatorial aggregate active set.  Exact enumeration
would require all signed \(k\)-subsets, so the table compares a random
aggregate vertex, a greedy full-space aggregate rule, and the same greedy rule
using only the sketched norms.  The setting \(k=50\) is used for all three
data sets, and the sketched rule uses the finite-horizon budget
\eqref{eq:mkbudget} with \(d=n\), \(C=1\), \(\eta=0.8\), and \(\delta=0.05\).

\begin{table}[!ht]
\centering
\caption{Top-\(k\) support-function benchmark \eqref{eq:topksupport} with
\(k=50\).  Here ``samples'' is \(N\) and ``features'' is the variable
dimension \(n\).  The norm ratio is relative to the greedy full-space
aggregate reference, not to the unavailable exact enumeration of all signed
\(k\)-subsets.}
\label{tab:topksupport}
\scriptsize
\resizebox{\linewidth}{!}{\begin{tabular}{lrrrlrrrr}
\toprule
data & samples & features & $k$ & method & dirs & objective & norm ratio & CPU s\\
\midrule
a8a & 22696 & 123 & 50 & centered DCA & -- & 0 & 0.000 & 0.00334\\
a8a & 22696 & 123 & 50 & random aggregate & -- & -3.23e+03 & 0.156 & 0.00194\\
a8a & 22696 & 123 & 50 & full-space greedy & -- & -1.47e+04 & 1.000 & 0.016\\
a8a & 22696 & 123 & 50 & RA sketched greedy & 203 & -1.49e+04 & 1.001 & 0.0923\\
\midrule
phishing & 11055 & 68 & 50 & centered DCA & -- & 0 & 0.000 & 0.000357\\
phishing & 11055 & 68 & 50 & random aggregate & -- & -234 & 0.139 & 0.000325\\
phishing & 11055 & 68 & 50 & full-space greedy & -- & -1.1e+03 & 1.000 & 0.0106\\
phishing & 11055 & 68 & 50 & RA sketched greedy & 118 & -1.13e+03 & 1.031 & 0.0211\\
\midrule
w8a & 45546 & 300 & 50 & centered DCA & -- & 0 & 0.000 & 0.000741\\
w8a & 45546 & 300 & 50 & random aggregate & -- & -2.59e+03 & 0.052 & 0.000605\\
w8a & 45546 & 300 & 50 & full-space greedy & -- & -9.11e+04 & 1.000 & 0.0399\\
w8a & 45546 & 300 & 50 & RA sketched greedy & 480 & -9.09e+04 & 0.996 & 0.344\\
\bottomrule
\end{tabular}

}
\end{table}

Table \ref{tab:topksupport} is more representative of the block difficulty
than Table \ref{tab:libsvmsupport}: the active set is no longer a list of
single rows.  Random aggregate vertices are much weaker, especially on
\texttt{w8a}, where their norm ratio is only \(0.052\).  The RA sketched
greedy rule recovers essentially the same aggregate quality as the full-space
greedy reference, with norm ratios \(1.001\), \(1.031\), and \(0.996\) on the
three data sets.  The ratio can exceed one because both rows are greedy
heuristics for a combinatorial aggregate search; the point is that the sketch
tracks the full-space greedy choice closely while avoiding explicit
enumeration of all signed \(k\)-subsets.

\subsection{Trimmed sparse-regression benchmark}
\label{sec:trimmedreg}

The same sparse matrices also define a robust-estimation test.  Least trimmed
squares and sparse least trimmed squares are classical tools for regression
with outliers \cite{alfons2013sparse}.  With binary labels
\(y_i\in\{-1,1\}\), consider
\begin{equation}\label{eq:trimmedreg}
        F_{\rm tr}(w)
        =
        \frac{1}{2N}\norm{Xw-y}^2
        +\frac{\lambda}{2}\norm{w}^2
        -
        \max_{\substack{S\subseteq\{1,\ldots,N\}\\ |S|=q}}
        \frac{1}{2N}\sum_{i\in S}(a_i^\top w-y_i)^2 .
\end{equation}
Equivalently, \eqref{eq:trimmedreg} keeps the \(N-q\) smallest squared
residuals and adds a ridge term.  Starting from \(w^0=0\), all squared
residuals are equal to one, so every \(q\)-subset is active.  Thus the first
DCA step faces a combinatorial aggregate active set.  For a chosen subset
\(S^k\), the active gradient of the concave term is
\[
        v^k=
        \frac1N X_{S^k}^{\top}(X_{S^k}w^k-y_{S^k}),
\]
and the convex subproblem is the linear system
\[
        \left(\frac1N X^\top X+\lambda I\right)w^{k+1}
        =
        \frac1N X^\top y+v^k .
\]
The experiment uses \(q=500\), \(\lambda=10^{-2}\), and one DCA step from
\(w^0=0\).  The random row averages ten random \(q\)-subsets.  The full-space
and sketched greedy rows approximate the aggregate residual search in the
original space and the sketched space.

\begin{table}[!ht]
\centering
\caption{Trimmed sparse-regression benchmark \eqref{eq:trimmedreg} with
\(q=500\) and \(\lambda=10^{-2}\).  Here ``samples'' is \(N\) and
``features'' is the variable dimension \(n\).  ``Trimmed MSE'' is the average
squared residual over the retained \(N-q\) samples.}
\label{tab:trimmedreg}
\scriptsize
\resizebox{\linewidth}{!}{\begin{tabular}{lrrrlrrrr}
\toprule
data & samples & features & trim & method & dirs & objective & trimmed MSE & CPU s\\
\midrule
a8a & 22696 & 123 & 500 & centered aggregate & -- & 0.2018 & 0.4039 & 0.0177\\
a8a & 22696 & 123 & 500 & random aggregate & -- & 0.2018 & 0.4039 & 0.00196\\
a8a & 22696 & 123 & 500 & full-space greedy & -- & 0.201 & 0.4023 & 0.125\\
a8a & 22696 & 123 & 500 & RA sketched greedy & 207 & 0.2013 & 0.4025 & 0.556\\
\midrule
phishing & 11055 & 68 & 500 & centered aggregate & -- & 0.009928 & 0.004196 & 0.00247\\
phishing & 11055 & 68 & 500 & random aggregate & -- & 0.009964 & 0.004203 & 0.00145\\
phishing & 11055 & 68 & 500 & full-space greedy & -- & 0.01235 & 0.01065 & 0.0945\\
phishing & 11055 & 68 & 500 & RA sketched greedy & 121 & 0.01261 & 0.01037 & 0.0823\\
\midrule
w8a & 45546 & 300 & 500 & centered aggregate & -- & 0.1525 & 0.2944 & 0.00408\\
w8a & 45546 & 300 & 500 & random aggregate & -- & 0.1525 & 0.2945 & 0.00361\\
w8a & 45546 & 300 & 500 & full-space greedy & -- & 0.1513 & 0.2936 & 0.27\\
w8a & 45546 & 300 & 500 & RA sketched greedy & 484 & 0.1515 & 0.2942 & 2.58\\
\bottomrule
\end{tabular}

}
\end{table}

Table \ref{tab:trimmedreg} is included as a boundary test rather than
as evidence of uniform improvement.  On \texttt{a8a} and \texttt{w8a}, the sketched
greedy rule tracks the full-space greedy rule and improves the trimmed
objective relative to centered and random aggregate choices.  On
\texttt{phishing}, however, the same stationarity-residual heuristic gives a
larger first-step objective.  This does not conflict with the theory:
\eqref{eq:trimmedreg} is a block sum-of-max model, and the table uses a
greedy aggregate heuristic rather than the exact fixed-factor aggregate
safeguard in Corollary \ref{cor:approxblock}.  The result marks a
computational boundary.  In block robust-regression
models, the sketched residual can reproduce the full-space aggregate heuristic,
but objective-quality improvements depend on whether that heuristic is aligned
with the statistical loss.

\subsection{Synthetic complementarity and binary-penalty blocks}
\label{sec:binarycomp}
\label{sec:lcp}

Binary restrictions provide the simplest complementarity block:
\[
        x_i\in\{0,1\}
        \quad\Longleftrightarrow\quad
        0\le x_i\perp 1-x_i\ge0
        \quad (0\le x_i\le1).
\]
The corresponding penalty
\[
        \sum_i \min\{x_i,1-x_i\}
        =
        \frac n2-\sum_i\max\{x_i-\tfrac12,\tfrac12-x_i\}
\]
has exactly the two-piece max structure that produces nonstationary critical
points at \(x_i=1/2\).  The fully separable case is therefore a useful baseline
test but is too symmetric to be a meaningful performance benchmark: any
active vertex moves a centered coordinate to a box endpoint.  The reported
experiment uses a less symmetric complementarity penalty.

Let \(M\ge0\) be a
sparse row-stochastic matrix and \(y=Mx\).  On the box \(0\le x\le1\), consider
\begin{equation}\label{eq:lcppenalty}
        \min_{0\le x\le1}\ 
        \frac12\norm{x-c}^2+\rho\sum_{i=1}^n \min\{x_i,y_i\},
        \qquad c=\tfrac12\one,\quad \rho=1 .
\end{equation}
This is a simple MPEC/LCP-style complementarity penalty.  Since
\[
        \min\{x_i,y_i\}=x_i+y_i-\max\{x_i,y_i\},
\]
the concave part is a sum of two-piece max blocks.  For a chosen aggregate
active gradient \(v^k\), the DCA subproblem has the explicit solution
\[
        x^{k+1}=
        \Pi_{[0,1]^n}\left(c-\rho(\one+M^\top\one)+v^k\right).
\]
We generate eight random sparse row-stochastic matrices for each dimension,
with five nonzeros per row.  The centered rule averages tied block gradients,
and the random rule chooses one tied block gradient uniformly.  The full-space
rule uses greedy aggregate choices in \(\R^n\), while RA applies the same
greedy search after sketching the block gradients.

\begin{table}[!ht]
\centering
\caption{Synthetic LCP/MPEC penalty benchmark \eqref{eq:lcppenalty}.  The
columns ``min gap'' and ``product gap'' report
\(n^{-1}\sum_i\min\{x_i,(Mx)_i\}\) and
\(n^{-1}\sum_i |x_i(Mx)_i|\), respectively.}
\label{tab:lcp}
\scriptsize
\resizebox{\linewidth}{!}{\begin{tabular}{rlrrrrrr}
\toprule
$n$ & method & dirs & objective & min gap & product gap & iterations & CPU s\\
\midrule
100 & centered DCA & -- & 12.5 & 0 & 0 & 2.0 & 0.00211\\
100 & random block & -- & 9.72 & 0.00907 & 0.00355 & 5.8 & 0.0011\\
100 & full-space greedy & -- & 9.62 & 0.0138 & 0.00509 & 4.1 & 0.00185\\
100 & RA sketched greedy & 166 & 9.68 & 0.0101 & 0.00376 & 4.1 & 0.00307\\
\midrule
300 & centered DCA & -- & 37.5 & 0 & 0 & 2.0 & 0.00333\\
300 & random block & -- & 29.6 & 0.00996 & 0.00391 & 6.9 & 0.00447\\
300 & full-space greedy & -- & 28.2 & 0.013 & 0.00484 & 4.4 & 0.00751\\
300 & RA sketched greedy & 479 & 28.7 & 0.0113 & 0.00432 & 5.8 & 0.0183\\
\midrule
500 & centered DCA & -- & 62.5 & 0 & 0 & 2.0 & 0.0046\\
500 & random block & -- & 48.8 & 0.0102 & 0.00396 & 7.9 & 0.0123\\
500 & full-space greedy & -- & 47.1 & 0.0125 & 0.00475 & 5.1 & 0.0102\\
500 & RA sketched greedy & 791 & 48 & 0.0117 & 0.00446 & 5.4 & 0.0351\\
\bottomrule
\end{tabular}

}
\end{table}

Table \ref{tab:lcp} should be read as a block-safeguard baseline test rather
than a general MPEC benchmark.  The centered rule quickly
drives the complementarity gaps to zero by collapsing to the trivial
complementary point, but this gives a larger penalty objective.  The greedy
aggregate rules retain a better balance with the quadratic term and lower the
objective substantially.  The RA sketched greedy rule stays close to the
full-space greedy reference, with a modest additional cost from forming the
sketch.  This supports the interpretation of the QUBO and binary-penalty
discussion: for block penalties, the computational question is how well a tractable
aggregate search approximates the exact safeguard in
Corollary \ref{cor:approxblock}.

\subsection{OR-Library QUBO benchmarks and Gurobi reference}
\label{sec:qubo}

We also test a less degenerate binary quadratic model.  Unconstrained binary
quadratic programming is a standard combinatorial benchmark class
\cite{kochenberger2014qubo} and appears naturally in Boolean polynomial
optimization tests \cite{niu2022boolean}.  The instances are thirty
OR-Library UBQP problems \cite{beasley1990orlibrary}, namely
\(bqp50.1\)--\(bqp50.10\), \(bqp100.1\)--\(bqp100.10\), and
\(bqp250.1\)--\(bqp250.10\).  The original OR-Library files are maximization
instances; we use the equivalent minimization QUBO form and compare with the
best-known values reported in the Biq Mac library \cite{wiegele2007biqmac}:
\begin{equation}\label{eq:qubo}
        \min_{z\in\{0,1\}^n} z^\top Qz,\qquad n\in\{50,100,250\}.
\end{equation}
The RA-DCA-type methods are applied to the box-penalized relaxation
\begin{equation}\label{eq:qubopenalty}
        \min_{0\le x\le1}\ 
        x^\top Qx+\rho\sum_{i=1}^n\min\{x_i,1-x_i\},
        \qquad \rho=1.
\end{equation}
For an indefinite \(Q\), we test two DC decompositions
\[
        Q=Q_+-Q_-,\qquad Q_+\succeq0,\quad Q_-\succeq0.
\]
The first is the shift split \(Q_+=Q+\gamma I\), \(Q_-=\gamma I\), with
\(\gamma>-\lambda_{\min}(Q)\).  The second is the spectral split obtained by
separating the positive and negative eigenvalues of \(Q\).  In both cases
\eqref{eq:qubopenalty} becomes
\[
        g(x)-h(x)
        =x^\top Q_+x
        -\left(x^\top Q_-x+\rho\sum_i
        \max\{x_i-\tfrac12,\tfrac12-x_i\}\right)
        +\frac{\rho n}{2}.
\]
Given a block sign vector
\(s^k_i\in\partial\max\{x_i-\frac12,\frac12-x_i\}\), the DCA subproblem is the
bound-constrained convex quadratic program
\begin{equation*}
        \min_{0\le x\le1}\ 
        x^\top Q_+x-\left(2Q_-x^k+\rho s^k\right)^\top x,
\end{equation*}
or, equivalently,
\[
        \min_{0\le x\le1}\ \frac12 x^\top Hx+f^\top x,
        \quad H=2Q_+,\quad f=-(2Q_-x^k+\rho s^k).
\]
This problem is not separable unless \(Q_+\) is diagonal.  The inner
projected-gradient iteration uses the box projection
\[
        x\leftarrow \Pi_{[0,1]^n}\{x-\alpha(Hx+f)\}
        =\min\{1,\max\{0,x-\alpha(Hx+f)\}\}.
\]
The QUBO implementation therefore uses this projected-gradient iteration, with
Nesterov acceleration and an active-set polish, instead of calling a generic
QP solver.  The Gurobi reference solve uses the Gurobi MATLAB interface when
available.  The DC decomposition \(Q_+-Q_-\) is computed once per instance and
is not included in the reported CPU time.  In particular, the spectral
eigenvalue decomposition is excluded; the larger spectral runtimes in
Table \ref{tab:quboorlib} come
from the subsequent box-QP solves, since the spectral split forms dense
\(Q_+\) and \(Q_-\), whereas the shift split preserves the sparsity of \(Q\)
up to a diagonal shift.  The centered method uses only
\(x^0=\tfrac{1}{2}\one\) in the main OR-Library comparison.  The
random-vertex and RA-DCA block rows in that comparison use multistart budgets
of \(40\), \(60\), and \(80\) for \(n=50\), \(100\), and \(250\), respectively;
the first start is \(\tfrac{1}{2}\one\), and the remaining starts are
independent uniform points in \([0,1]^n\).  Only RA-DCA block uses sampled
directions.  For the main OR-Library rows, we set their number by the
finite-horizon version of
\eqref{eq:mkbudget}, using the block bound \(d=n\), \(C=1\), \(\eta=0.8\),
\(\delta=0.05\), and \(K=60\) times the number of starts.  This gives
\(m=95\), \(174\), and \(409\) for \(n=50\), \(100\), and \(250\),
respectively.  The constant in the embedding theorem is not optimized; this
choice is a direct normalized implementation of the \(m=O(n+\log(K/\delta))\)
scaling.  The reported objective is the original binary QUBO objective at the
rounded point \(z_i=\mathbf 1_{\{x_i\ge1/2\}}\).

In this QUBO comparison, Gurobi \cite{gurobi2026} is used only as a
reference global MIQP solver for \eqref{eq:qubo}, with time limits of \(10\),
\(15\), and \(20\) seconds for \(n=50\), \(100\), and \(250\), respectively.
These are fixed calibration caps, chosen so that the branch-and-cut reference
does not dominate the experiment.  Its runtime is not directly comparable with
the DCA heuristics, since it solves the binary problem rather than the penalty
relaxation.  Longer Gurobi runs would strengthen certificates by reducing MIP
gaps, but they would not change the DCA objective values reported in
Table \ref{tab:quboorlib}.
In Table \ref{tab:quboorlib} and Table \ref{tab:quboinstances} in Appendix
\ref{sec:quboappendix}, ``Gurobi MIP gap'' is the relative optimality gap
reported by Gurobi,
\[
        100\,\frac{|\mathrm{ObjBound}-\mathrm{ObjVal}|}
        {|\mathrm{ObjVal}|},
\]
where \(\mathrm{ObjVal}\) is the incumbent binary objective and
\(\mathrm{ObjBound}\) is Gurobi's best global bound at termination.  Small
nonzero entries, such as \(0.01\%\), can occur even when Gurobi reports
\texttt{OPTIMAL}, because the certificate is closed only to the solver's
optimality tolerance; this does not mean that the incumbent objective differs
from the best-known value.

\begin{table}[!ht]
\centering
\caption{OR-Library UBQP/QUBO results on \(bqp50\), \(bqp100\), and
\(bqp250\).  The gap columns report percentage gaps of the rounded binary
objective relative to the best-known value.  Decomposition time is excluded
from CPU s.  The column ``dirs'' applies only to RA-DCA block.  The last
column reports Gurobi's solver optimality gap at the time limit and is
therefore filled only on the Gurobi rows.}
\label{tab:quboorlib}
\scriptsize
\resizebox{\linewidth}{!}{\begin{tabular}{rlrrrrrrr}
\toprule
$n$ & method & starts & dirs & mean gap \% & max gap \% & hit rate & CPU s & Gurobi MIP gap \%\\
\midrule
\multirow{7}{*}{50} & centered DCA (shift) & 1 & -- & 0.37 & 2.37 & 0.60 & 0.0124 & --\\
 & random-vertex DCA (shift) & 40 & -- & 0.00 & 0.00 & 1.00 & 0.172 & --\\
 & RA-DCA block (shift) & 40 & 95 & 0.00 & 0.00 & 1.00 & 0.175 & --\\
 & centered DCA (spectral) & 1 & -- & 1.45 & 5.39 & 0.30 & 0.00588 & --\\
 & random-vertex DCA (spectral) & 40 & -- & 0.06 & 0.57 & 0.90 & 0.234 & --\\
 & RA-DCA block (spectral) & 40 & 95 & 0.06 & 0.57 & 0.90 & 0.222 & --\\
 & Gurobi & -- & -- & 0.00 & 0.00 & 1.00 & 0.017 & 0.00\\
\midrule
\multirow{7}{*}{100} & centered DCA (shift) & 1 & -- & 0.48 & 1.88 & 0.50 & 0.00652 & --\\
 & random-vertex DCA (shift) & 60 & -- & 0.09 & 0.80 & 0.70 & 0.373 & --\\
 & RA-DCA block (shift) & 60 & 174 & 0.07 & 0.65 & 0.80 & 0.379 & --\\
 & centered DCA (spectral) & 1 & -- & 1.96 & 4.12 & 0.10 & 0.0219 & --\\
 & random-vertex DCA (spectral) & 60 & -- & 0.14 & 1.10 & 0.70 & 1.33 & --\\
 & RA-DCA block (spectral) & 60 & 174 & 0.14 & 1.00 & 0.80 & 1.32 & --\\
 & Gurobi & -- & -- & 0.00 & 0.00 & 1.00 & 0.184 & 0.00\\
\midrule
\multirow{7}{*}{250} & centered DCA (shift) & 1 & -- & 0.58 & 1.36 & 0.10 & 0.0173 & --\\
 & random-vertex DCA (shift) & 80 & -- & 0.25 & 0.58 & 0.10 & 1.51 & --\\
 & RA-DCA block (shift) & 80 & 409 & 0.26 & 0.69 & 0.20 & 1.5 & --\\
 & centered DCA (spectral) & 1 & -- & 1.62 & 4.19 & 0.00 & 0.177 & --\\
 & random-vertex DCA (spectral) & 80 & -- & 0.48 & 2.26 & 0.10 & 16.8 & --\\
 & RA-DCA block (spectral) & 80 & 409 & 0.46 & 2.25 & 0.10 & 16.8 & --\\
 & Gurobi & -- & -- & 0.00 & 0.00 & 1.00 & 20.2 & 15.72\\
\bottomrule
\end{tabular}

}
\end{table}

\begin{figure}[t]
\centering
\includegraphics[width=.49\linewidth]{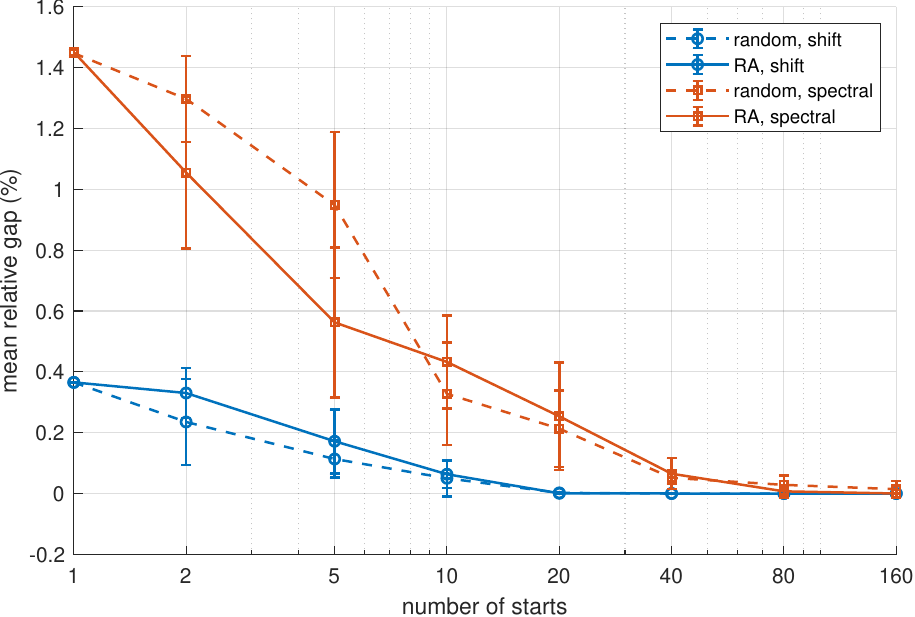}
\includegraphics[width=.49\linewidth]{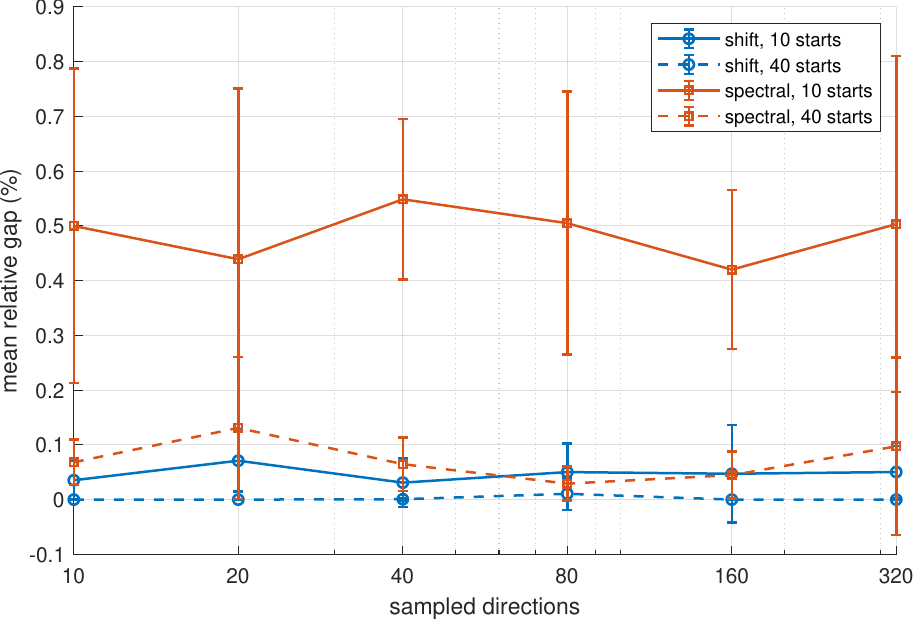}
\caption{QUBO budget sensitivity on the OR-Library \(bqp50.1\)--\(bqp50.10\)
instances.  Left: varying the number of starts with \(m=80\) directions for
RA-DCA block.  Right: varying the number of sampled directions for RA-DCA
block with ten and forty starts.  Error bars show one standard deviation over
eight random seeds.}
\label{fig:qubobudget}
\end{figure}

Table \ref{tab:quboorlib} shows that the QUBO behaviour is more nuanced than
the fully symmetric binary-complementarity test, and the larger instances are
no longer saturated by the local methods.  On \(bqp50\), the shift split with
forty starts reaches all ten best-known values for both multistart methods.
On \(bqp100\), the embedding-style direction budget improves the RA-shift mean
gap to \(0.07\%\), compared with \(0.09\%\) for random-vertex DCA, and gives a
\(0.80\) hit rate.  The same budget also raises the RA-spectral hit rate to
\(0.80\).  On \(bqp250\), the problem is harder and the larger direction
budget is not a cure by itself: RA-shift obtains a \(0.26\%\) mean gap and a
\(0.20\) hit rate, while random-vertex DCA has a \(0.25\%\) mean gap and a
\(0.10\) hit rate.  The spectral split is more expensive because the
projected-gradient QP is harder; RA-spectral slightly improves the mean gap
relative to random-vertex DCA (\(0.46\%\) versus \(0.48\%\)), but not the hit
rate.

Figure \ref{fig:qubobudget} gives a broader view of the two most natural
budget parameters on the \(bqp50\) group.  Increasing the number of starts
steadily improves both random-vertex DCA and RA-DCA block, for both
decompositions.  With forty starts the shift split reaches the best-known
value on all ten \(bqp50\) instances in all eight random repeats, while
RA-DCA block with the spectral split falls
below a \(0.07\%\) mean gap.  Increasing the starts further continues to help
the spectral split: in this sweep its mean gap drops to about \(0.007\%\) at
eighty starts and to zero at 160 starts.  In contrast, increasing the number
of sampled directions has only a small and nonmonotone effect.  The same
pattern remains when the direction sweep is repeated with forty starts: the
curves are lower because the multi-start budget is larger, but there is still
no monotone improvement as \(m\) grows.  This is expected for the binary
penalty: after a random continuous start, most coordinates are not tied at
\(x_i=1/2\), so the RA direction rule mainly affects the centered start and
does not substitute for multi-start exploration.  For QUBO, the additional
budget is therefore more effective when assigned first to starts; increasing
directions is a secondary tuning knob.
This also explains why RA-DCA block is not uniformly dominant on QUBO.  After
random continuous starts, most coordinates quickly lie away from \(x_i=1/2\),
and the binary penalty supplies an essentially unique block subgradient.  At
such iterates there is little active-set ambiguity for the randomized rule to
exploit, so the quality is governed more by the start pool, the DC split, and
the box-QP dynamics than by the direction sketch.

Appendix \ref{sec:quboappendix} gives the instancewise gaps, gap profiles,
spectral-split tuning diagnostics, and a separate \(bqp250.8\) backend
diagnostic.  Those tables support the same interpretation but are separated
from the main text because they refine the QUBO boundary case rather than
establish the finite-max active-set mechanism itself.

\subsection{Computational interpretation}

The single finite-max experiments are the main computational tests of the
theory: the sampled vertex safeguard recovers near-full-vertex active
gradients and avoids nonstationary averaged-subgradient criticality at
controlled scale.  The
top-\(k\) support experiment gives the clearest block result, because exact
aggregate enumeration is combinatorial while the sketched greedy rule closely
tracks the full-space greedy reference.  The remaining boundary tests
show that active-set screening is not the only determinant of performance:
statistical objective alignment matters in trimmed regression, and the QUBO
results are strongly affected by starts and the DC decomposition.  Thus the
computational claim is specific: RA-DCA strengthens active-subgradient
selection at nonsmooth ties or near-ties; it is not a replacement for
problem-specific global or multi-start strategies.

\section{Conclusions}
\label{sec:conclusion}

RA-DCA is based on the following observation: for finite-max DC programs,
active-gradient convex combinations are useful computationally but are not by
themselves directional-stationarity certificates.  The proposed method keeps
the inexpensive DCA subproblem, uses randomized active-set projections to
screen active vertices, and calls the convex-combination LP only as a
low-residual centering fallback.  The vertex safeguard is the mechanism that
restores a provable directional-stationarity guarantee.  The block corollary
shows how the same stationarity certificate extends to finite sums of max
terms when
aggregate active vertices can be safeguarded.

The dominant linear algebra is a product between random directions and active
gradients, making the active-set screening layer suitable for vectorized
implementation and, when the surrounding subproblem solver has enough
arithmetic intensity, GPU acceleration.  The clearest computational evidence
comes from settings where the active set itself is the bottleneck: degenerate
finite maxima, sparse support-function models, and top-\(k\) aggregate
selection.  In these cases RA-DCA addresses the targeted stationarity gap:
averaged active subgradients can certify only criticality, while a sampled
vertex residual exposes the remaining directional-stationarity violation.

The mixed QUBO and trimmed-regression results also indicate the boundary of
the method: when iterates spend little time at nonsmooth ties, or when starts
and the DC decomposition determine the solution quality, active-set selection
is secondary.  For binary and QUBO-type penalty models, a natural algorithmic
extension is to combine the randomized active-set selection layer with
adaptive DC decompositions, penalty-continuation strategies, restart or
multistart policies, and specialized bound-constrained QP solvers followed by
rounding and polishing.  Promising extensions also include larger block
sum-of-max applications where active-set ambiguity remains central, such as
MIPLIB-derived penalty instances \cite{gleixner2021miplib}, clustering models,
and implicit top-\(k\) or robust-loss formulations.  Other directions include
acceleration techniques, such as inertial or extrapolated DCA steps
\cite{wen2018proximal}, adaptive sampling budgets, and batched CPU/GPU kernels
for large structured subproblems, as well as stochastic and variance-reduced
variants \cite{reddi2016stochastic} for empirical-risk and machine-learning
models, where both data points and active max blocks may be sampled; such
extensions would require new convergence analysis for stochastic active sets
and inexact subproblem solves.

\appendix\normalsize

\section{Worst-Iterate Residual Bound}
\label{sec:residualcomplexity}

\begin{proposition}[Worst-iterate residual bound]\label{prop:resrate}
Suppose Assumptions \ref{ass:basic} and \ref{ass:embed} hold, and work on
iterations for which the embedding event \eqref{eq:embedding} holds.  Let
\(F_{\inf}\) be a lower bound of \(F\),
\(\Delta_0=F(x^0)-F_{\inf}\),
\(E_N=\sum_{k=0}^N\eps_k\), and
\[
        C_\eta=\frac{(1+\eta)(L_g+\sigma)}{1-\eta}.
\]
Then, for every such iteration \(k\),
\begin{equation}\label{eq:resstepbound}
        R_{\rm d}(x^k)
        \le
        \max\left\{
        \frac{\tau_k}{1-\eta},
        C_\eta\norm{x^{k+1}-x^k}
        \right\}.
\end{equation}
This gives the ergodic step estimate:
\begin{equation}\label{eq:steprate}
        \min_{0\le k\le N}\norm{x^{k+1}-x^k}
        \le
        \left(\frac{2(\Delta_0+E_N)}{\mu(N+1)}\right)^{1/2}.
\end{equation}
If, in addition, \(\tau_k\) is nonincreasing, then
\begin{equation}\label{eq:dresrate}
        \min_{\lfloor N/2\rfloor\le k\le N}R_{\rm d}(x^k)
        \le
        \max\left\{
        \frac{\tau_{\lfloor N/2\rfloor}}{1-\eta},
        C_\eta
        \left(\frac{4(\Delta_0+E_N)}{\mu(N+1)}\right)^{1/2}
        \right\}.
\end{equation}
\end{proposition}

\begin{proof}
Fix an iteration on which \eqref{eq:embedding} holds.  Since
\(\eps_k\ge0\), the exact active set \(\cA(x^k)\) is contained in
\(\cA_k\).  If the safeguard is not triggered, then
\(\widehat R_k\le\tau_k\), and the lower embedding bound gives
\[
        (1-\eta)R_{\rm d}(x^k)\le \widehat R_k\le\tau_k .
\]
If the safeguard is triggered, \(v^k=\nabla\psi_{i_k}(x^k)\) and
\(\widehat R_k=\norm{D_k(v^k-\nabla g(x^k))}\).  Since \(i_k\) maximizes the
sampled vertex residual, the lower embedding bound gives
\((1-\eta)R_{\rm d}(x^k)\le\widehat R_k\).  The upper embedding bound and the
first-order condition for \eqref{eq:update} give
\[
        \widehat R_k
        \le (1+\eta)\norm{v^k-\nabla g(x^k)}
        \le (1+\eta)(L_g+\sigma)\norm{x^{k+1}-x^k}.
\]
This proves \eqref{eq:resstepbound}.  Inequality \eqref{eq:steprate} follows
by summing the numerical-active-set descent estimate \eqref{eq:descent}.
Applying the same summability bound to the block
\(\{\lfloor N/2\rfloor,\ldots,N\}\) gives an index in that block with
\[
        \norm{x^{k+1}-x^k}
        \le \left(\frac{4(\Delta_0+E_N)}{\mu(N+1)}\right)^{1/2}.
\]
Using \eqref{eq:resstepbound} at this index and the monotonicity of
\(\tau_k\) proves \eqref{eq:dresrate}.
\qed
\end{proof}

\begin{remark}
Proposition \ref{prop:resrate} is a worst-iterate complexity statement for
stationarity.  It shows that, if \(\tau_k=O(k^{-1/2})\) or smaller and the
sampling budget enforces the embedding event, the directional-stationarity
residual has the same \(O(k^{-1/2})\) worst-iterate decay that follows from
standard descent analyses for proximal DCA, up to the accumulated numerical
active-set error \(E_N\).  Under Assumption \ref{ass:basic}, \(E_N\) is bounded
uniformly in \(N\).
\end{remark}

\section{Approximate Block Safeguards and Inexact Subproblems}
\label{sec:appendixblock}

\begin{corollary}[Approximate aggregate safeguard]\label{cor:approxblock}
In the setting of Corollary \ref{cor:blockextension}, suppose that whenever
\(\widehat R_k^{\rm b}>\tau_k\) the block method selects an aggregate vertex
\(v^k\in\cV_k^\eps\) satisfying
\[
        \norm{D_k(v^k-\nabla g(x^k))}
        \ge \theta\,\widehat R_k^{\rm b}
\]
for some fixed \(\theta>0\).  Then every accumulation point is
directional-stationary for \eqref{eq:blockmodel} with probability one.
\end{corollary}

\begin{proof}
The proof is the same as that of Corollary \ref{cor:blockextension}.  If
\(R_{\rm b}(\bar x)>0\), then along a subsequence
\(\widehat R_k^{\rm b}\ge (1-\eta)\rho\) for some \(\rho>0\).  The approximate
safeguard gives a selected aggregate vertex with sampled residual at least
\(\theta(1-\eta)\rho\).  The embedding upper bound then yields
\[
        \norm{v^k-\nabla g(x^k)}
        \ge \frac{\theta(1-\eta)}{1+\eta}\rho ,
\]
and the first-order condition again forces
\(\norm{x^{k+1}-x^k}\) to be bounded away from zero, contradicting descent.
\qed
\end{proof}

Corollary \ref{cor:approxblock} states the formal boundary between theory and
implementation for block heuristics.  A block implementation is covered by
the convergence proof if its aggregate-vertex search is an exact or
fixed-factor approximation to \eqref{eq:blocksampledres}.  Greedy block rules
without such a verified factor remain computational tests of the same
active-set geometry.

The same argument allows inexact convex subproblem solves.  If the computed
update satisfies
\[
        \nabla g(x^{k+1})-v^k+\sigma(x^{k+1}-x^k)=e^k,
        \qquad \norm{e^k}\to0,
\]
and the descent estimate remains valid up to a summable error, then the
almost-sure stationarity conclusion is unchanged.  For the residual bound in
Proposition \ref{prop:resrate}, a relative residual
\(\norm{e^k}\le\kappa\norm{x^{k+1}-x^k}\) simply replaces \(L_g+\sigma\) by
\(L_g+\sigma+\kappa\) in \(C_\eta\).

\section{Additional QUBO Diagnostics}
\label{sec:quboappendix}

This appendix collects QUBO diagnostics that refine the boundary-test
interpretation in Section \ref{sec:qubo}.  They are not needed for the main
finite-max claim, but they make clear where the QUBO behaviour is governed by
starts, the DC split, and implementation backend rather than by active-set
selection alone.

\begin{table}[!ht]
\centering
\caption{Instancewise QUBO gaps on the OR-Library groups \(bqp50\), \(bqp100\),
and \(bqp250\), separated by horizontal rules.  Entries C-shift, R-shift,
RA-shift, and RA-spec are percentage gaps relative to the best-known value;
Gurobi MIP gap is the solver optimality gap at the size-dependent time limit
\(10\), \(15\), or \(20\) seconds.}
\label{tab:quboinstances}
\scriptsize
\resizebox{\linewidth}{!}{\begin{tabular}{lrrrrrrr}
\toprule
instance & best & C-shift & R-shift & RA-shift & RA-spec & Gurobi MIP gap & Gurobi s\\
\midrule
bqp50.1 & -2098 & 0.00 & 0.00 & 0.00 & 0.00 & 0.00 & 0.022\\
bqp50.2 & -3702 & 0.00 & 0.00 & 0.00 & 0.00 & 0.00 & 0.017\\
bqp50.3 & -4626 & 0.00 & 0.00 & 0.00 & 0.00 & 0.00 & 0.018\\
bqp50.4 & -3544 & 2.37 & 0.00 & 0.00 & 0.00 & 0.00 & 0.014\\
bqp50.5 & -4012 & 0.00 & 0.00 & 0.00 & 0.00 & 0.00 & 0.015\\
bqp50.6 & -3693 & 0.00 & 0.00 & 0.00 & 0.00 & 0.00 & 0.024\\
bqp50.7 & -4520 & 0.22 & 0.00 & 0.00 & 0.00 & 0.00 & 0.002\\
bqp50.8 & -4216 & 0.00 & 0.00 & 0.00 & 0.57 & 0.00 & 0.018\\
bqp50.9 & -3780 & 1.01 & 0.00 & 0.00 & 0.00 & 0.00 & 0.02\\
bqp50.10 & -3507 & 0.06 & 0.00 & 0.00 & 0.00 & 0.00 & 0.02\\
\midrule
bqp100.1 & -7970 & 1.88 & 0.80 & 0.65 & 1.00 & 0.00 & 0.178\\
bqp100.2 & -11036 & 0.45 & 0.00 & 0.07 & 0.00 & 0.00 & 0.188\\
bqp100.3 & -12723 & 0.00 & 0.00 & 0.00 & 0.00 & 0.01 & 0.162\\
bqp100.4 & -10368 & 0.00 & 0.00 & 0.00 & 0.00 & 0.00 & 0.17\\
bqp100.5 & -9083 & 0.00 & 0.00 & 0.00 & 0.42 & 0.00 & 0.185\\
bqp100.6 & -10210 & 1.07 & 0.08 & 0.00 & 0.00 & 0.00 & 0.243\\
bqp100.7 & -10125 & 1.28 & 0.04 & 0.00 & 0.00 & 0.00 & 0.186\\
bqp100.8 & -11435 & 0.00 & 0.00 & 0.00 & 0.00 & 0.00 & 0.182\\
bqp100.9 & -11455 & 0.00 & 0.00 & 0.00 & 0.00 & 0.00 & 0.164\\
bqp100.10 & -12565 & 0.14 & 0.00 & 0.00 & 0.00 & 0.01 & 0.179\\
\midrule
bqp250.1 & -45607 & 0.26 & 0.11 & 0.00 & 0.16 & 9.82 & 20.2\\
bqp250.2 & -44810 & 1.36 & 0.16 & 0.16 & 0.68 & 14.89 & 20.2\\
bqp250.3 & -49037 & 0.34 & 0.00 & 0.15 & 0.05 & 9.96 & 20.2\\
bqp250.4 & -41274 & 0.55 & 0.27 & 0.09 & 0.36 & 14.60 & 20.2\\
bqp250.5 & -47961 & 0.58 & 0.15 & 0.05 & 0.25 & 13.38 & 20.2\\
bqp250.6 & -41014 & 0.99 & 0.58 & 0.69 & 0.59 & 21.86 & 20.2\\
bqp250.7 & -46757 & 0.00 & 0.00 & 0.00 & 0.00 & 17.08 & 20.2\\
bqp250.8 & -35726 & 0.56 & 0.56 & 0.56 & 2.25 & 25.38 & 20.2\\
bqp250.9 & -48916 & 0.53 & 0.44 & 0.37 & 0.20 & 10.93 & 20.2\\
bqp250.10 & -40442 & 0.60 & 0.26 & 0.50 & 0.05 & 19.27 & 20.2\\
\bottomrule
\end{tabular}

}
\end{table}

\begin{table}[!ht]
\centering
\caption{QUBO tuning check for RA-DCA block.  Each entry averages eight
independent random seeds on \(bqp50.1\)--\(bqp50.10\).  Decomposition time is
excluded.}
\label{tab:qubotuning}
\scriptsize
\resizebox{\linewidth}{!}{\begin{tabular}{lrrrrrr}
\toprule
configuration & \(\rho\) & dirs & starts & mean gap \% & std gap \% & hit rate\\
\midrule
shift, baseline & 1.00 & 80 & 10 & 0.087 & 0.095 & 0.875\\
spectral, baseline & 1.00 & 80 & 10 & 0.363 & 0.199 & 0.575\\
shift, larger penalty & 5.00 & 80 & 10 & 0.084 & 0.086 & 0.838\\
spectral, smaller penalty & 0.25 & 80 & 10 & 0.481 & 0.264 & 0.625\\
spectral, larger penalty & 5.00 & 80 & 10 & 0.538 & 0.360 & 0.563\\
spectral, 20 starts & 1.00 & 80 & 20 & 0.185 & 0.183 & 0.725\\
spectral, 40 starts & 1.00 & 80 & 40 & 0.093 & 0.119 & 0.900\\
shift, more directions & 1.00 & 320 & 10 & 0.034 & 0.042 & 0.888\\
\bottomrule
\end{tabular}

}
\end{table}

\begin{table}[!ht]
\centering
\caption{QUBO gap profile over all thirty OR-Library instances.  The first
four columns are percentage gap statistics relative to the best-known value;
the last three columns report the fraction of instances whose gap is below the
specified threshold.}
\label{tab:qubogapprofile}
\scriptsize
\begin{tabular}{lrrrrrrr}
\toprule
method & mean & median & std & max & $\le0.1\%$ & $\le0.5\%$ & $\le1\%$\\
\midrule
centered shift & 0.48 & 0.24 & 0.62 & 2.37 & 0.43 & 0.60 & 0.80\\
random shift & 0.11 & 0.00 & 0.21 & 0.80 & 0.70 & 0.90 & 1.00\\
RA shift & 0.11 & 0.00 & 0.21 & 0.69 & 0.77 & 0.87 & 1.00\\
centered spectral & 1.68 & 1.15 & 1.63 & 5.39 & 0.20 & 0.37 & 0.47\\
random spectral & 0.23 & 0.00 & 0.47 & 2.26 & 0.67 & 0.87 & 0.93\\
RA spectral & 0.22 & 0.00 & 0.46 & 2.25 & 0.67 & 0.83 & 0.93\\
\bottomrule
\end{tabular}

\end{table}

\begin{figure}[!ht]
\centering
\includegraphics[width=.78\linewidth]{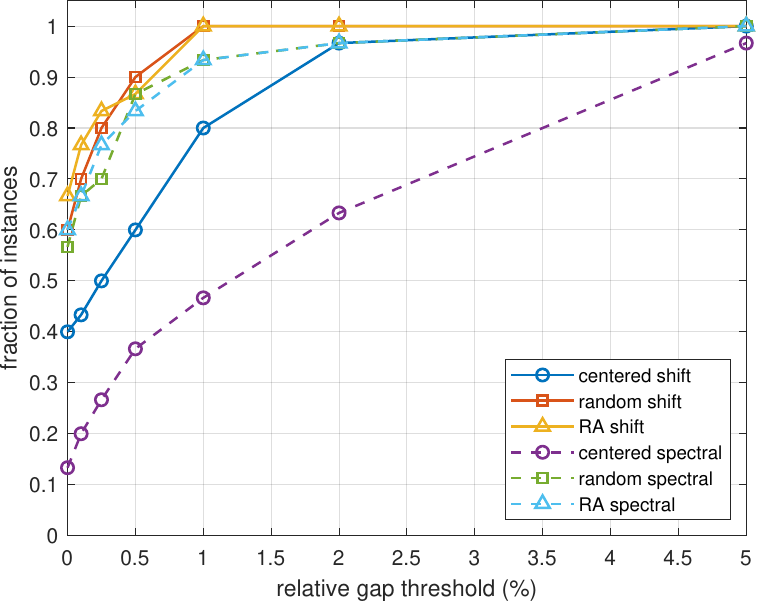}
\caption{QUBO gap profiles over all thirty OR-Library instances.  Each curve
reports the fraction of instances whose rounded binary objective is within the
given percentage gap from the best-known value.}
\label{fig:qubogapprofile}
\end{figure}

The gap profile gives a more cautious view than a single mean.  For the shift
split, random-vertex DCA and RA-DCA block have the same \(0.11\%\) mean gap,
but RA-DCA block improves the fraction of instances within \(0.1\%\) of the
best-known value from \(0.70\) to \(0.77\) and lowers the worst gap from
\(0.80\%\) to \(0.69\%\).  For the spectral split, the two vertex-based rules
are essentially tied in aggregate, which is consistent with the starts
sensitivity reported in Figure \ref{fig:qubobudget}.

\begin{table}[!ht]
\centering
\caption{Starts and backend sensitivity on \(bqp250.8\).  The table compares
serial CPU, batched CPU, and batched GPU prototypes using the same initial
multistart pool.  The speedup columns are relative to the matching serial CPU
or batched CPU row.}
\label{tab:qubobqp2508starts}
\scriptsize
\begingroup
\renewcommand{\arraystretch}{0.92}
\resizebox{0.96\linewidth}{!}{\begin{tabular}{llrrrrr}
\toprule
split & backend & starts & gap \% & time s & speedup-S & speedup-B\\
\midrule
shift & serial CPU & 20 & 0.560 & 1.69 & -- & --\\
shift & batched CPU & 20 & 0.560 & 2.23 & 0.76 & --\\
shift & batched GPU & 20 & 0.560 & 2.48 & 0.68 & 0.90\\
shift & serial CPU & 40 & 0.560 & 2.13 & -- & --\\
shift & batched CPU & 40 & 0.560 & 1.57 & 1.36 & --\\
shift & batched GPU & 40 & 0.560 & 2.99 & 0.71 & 0.52\\
shift & serial CPU & 80 & 0.560 & 4.07 & -- & --\\
shift & batched CPU & 80 & 0.560 & 2.6 & 1.56 & --\\
shift & batched GPU & 80 & 0.560 & 3.3 & 1.23 & 0.79\\
shift & serial CPU & 160 & 0.560 & 6.47 & -- & --\\
shift & batched CPU & 160 & 0.560 & 2.09 & 3.10 & --\\
shift & batched GPU & 160 & 0.560 & 1.72 & 3.76 & 1.21\\
shift & serial CPU & 256 & 0.498 & 7.13 & -- & --\\
shift & batched CPU & 256 & 0.498 & 6.67 & 1.07 & --\\
shift & batched GPU & 256 & 0.498 & 3.47 & 2.06 & 1.92\\
spectral & serial CPU & 20 & 3.101 & 9.13 & -- & --\\
spectral & batched CPU & 20 & 3.101 & 1.05 & 8.68 & --\\
spectral & batched GPU & 20 & 3.101 & 3.39 & 2.70 & 0.31\\
spectral & serial CPU & 40 & 2.220 & 15 & -- & --\\
spectral & batched CPU & 40 & 2.220 & 2.34 & 6.43 & --\\
spectral & batched GPU & 40 & 2.220 & 3.1 & 4.85 & 0.75\\
spectral & serial CPU & 80 & 2.069 & 30.5 & -- & --\\
spectral & batched CPU & 80 & 2.069 & 1.28 & 23.89 & --\\
spectral & batched GPU & 80 & 2.069 & 2.08 & 14.65 & 0.61\\
spectral & serial CPU & 160 & 1.360 & 60.1 & -- & --\\
spectral & batched CPU & 160 & 1.360 & 5.98 & 10.05 & --\\
spectral & batched GPU & 160 & 1.360 & 3.39 & 17.75 & 1.77\\
spectral & serial CPU & 256 & 1.050 & 96.5 & -- & --\\
spectral & batched CPU & 256 & 1.050 & 2.22 & 43.56 & --\\
spectral & batched GPU & 256 & 1.050 & 1.73 & 55.67 & 1.28\\
\bottomrule
\end{tabular}

}
\endgroup
\end{table}

\begin{figure}[!ht]
\centering
\includegraphics[width=.92\linewidth]{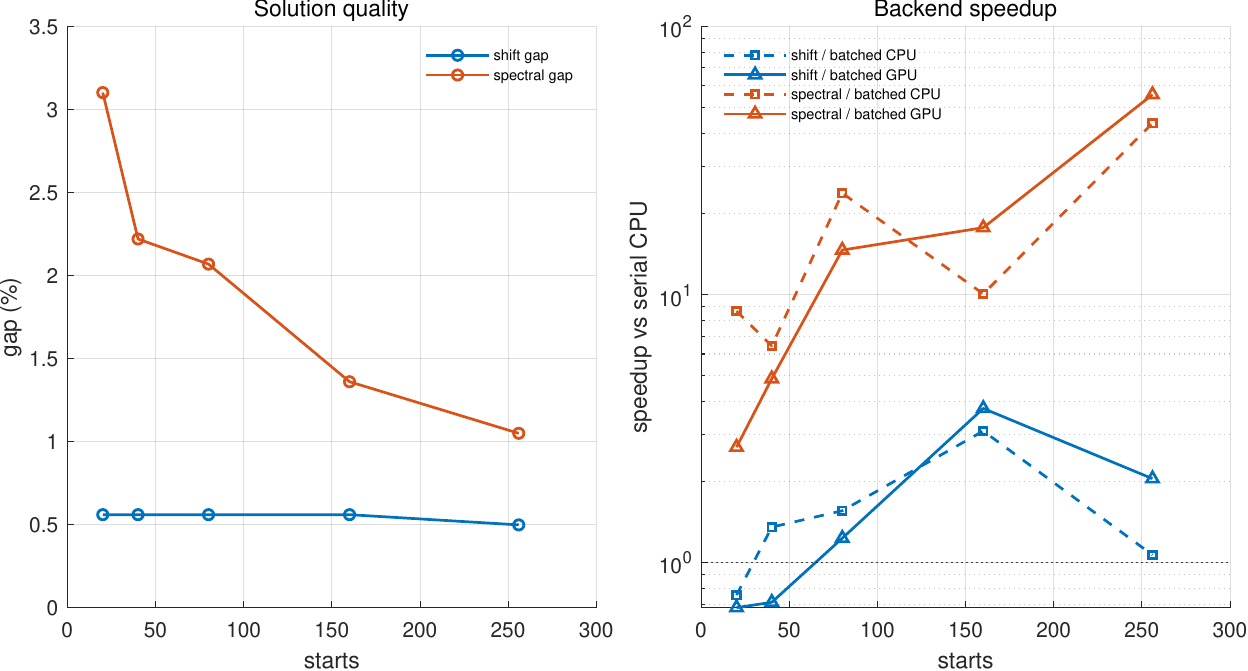}
\caption{Starts and backend sensitivity on \(bqp250.8\).  Left: percentage
gap versus the number of starts.  Right: speedup relative to the serial CPU
implementation; the horizontal line marks parity.  The batched implementations
solve all starts as a matrix of iterates, while the serial CPU implementation
follows the default per-start loop used in the main QUBO table.  For each
split and start budget, all three backends receive the same initial points.}
\label{fig:qubobqp2508starts}
\end{figure}

The \(bqp250.8\) diagnostic separates search quality from backend timing: for a
fixed split and start budget, all three backends use the same initial points
and therefore report the same gap.  The timing rows show that GPU acceleration
is useful only with a genuinely batched implementation; the serial per-start
GPU path is slowed by data movement and per-start kernel overhead.

\section{LP and Safeguard Diagnostics}
\label{sec:lpbackendappendix}

This appendix records diagnostics for the active convex-combination LP
\eqref{eq:lp}.  The first table reports how often the safeguard and LP branches
are used in the finite-max experiments.  A second diagnostic isolates a case
where the LP branch is useful: the numerical active set contains nearly active
but exactly inactive pieces.  The final table compares LP backend costs.
These diagnostics are separated from the main numerical section because they
validate the implementation costs and clarify the LP fallback, rather than
support a performance claim for an application class.

\begin{table}[!ht]
\centering
\caption{Safeguard frequency and LP use in the finite-max RA-DCA tests.  The
entries are means over ten random instances.  ``Default'' uses the parameters
of the reported experiments; ``LP-forced'' sets the safeguard tolerance to
\(+\infty\) to time the convex-combination branch.  ``Safeguard frac.'' is the
fraction of iterations using an active vertex, and ``LP calls'' is the mean
number of calls to \eqref{eq:lp} per run.}
\label{tab:lpsafeguard}
\scriptsize
\resizebox{\linewidth}{!}{\begin{tabular}{lrrlrrrrr}
\toprule
test & $n$ & pieces & setting & iter. & safeguard frac. & LP calls & LP ms & CPU s\\
\midrule
max-affine & 100 & 1000 & default & 2.0 & 1.00 & 0.0 & -- & 0.00904\\
max-affine & 100 & 1000 & LP-forced & 3.1 & 0.65 & 1.1 & 72.78 & 0.0826\\
max-affine & 500 & 5000 & default & 2.0 & 1.00 & 0.0 & -- & 0.0272\\
max-affine & 500 & 5000 & LP-forced & 3.0 & 0.67 & 1.0 & 433.74 & 0.462\\
max-quadratic & 100 & 1000 & default & 18.0 & 1.00 & 0.0 & -- & 0.00723\\
max-quadratic & 100 & 1000 & LP-forced & 23.1 & 0.88 & 5.2 & 70.29 & 0.397\\
max-quadratic & 500 & 5000 & default & 18.0 & 1.00 & 0.0 & -- & 0.0386\\
max-quadratic & 500 & 5000 & LP-forced & 19.0 & 0.95 & 1.0 & 429.77 & 0.473\\
\bottomrule
\end{tabular}

}
\end{table}

The default rows show that, on the degenerate signed-pair tests, the sampled
vertex residual remains large until the method reaches a single active piece,
so the LP is skipped.  The LP-forced rows time the convex-combination branch on
the same active sets and are not used for the main performance claims here.

The following one-dimensional diagnostic shows when the LP does change the
computed step.  Consider
\begin{align*}
        F(x)
        &=
        \frac12x^2-
        \max\{0,\ c_1x-\delta,\ c_2x-\delta,\ c_3x-\delta\},\\
        (c_1,c_2,c_3)&=(0.010,0.015,0.020),
\end{align*}
with \(\delta=3\cdot10^{-4}\).  At \(x^0=0\), the flat piece is the only
exactly active piece and \(x=0\) is an exact directional-stationary minimizer.
However, with \(\eps_0=4\cdot10^{-4}\), all four pieces enter the numerical
active set.  A vertex step can then move to an exactly stationary but larger
local point generated by an inactive piece.  With \(\tau_0=2.5\cdot10^{-2}\)
and decreasing \(\eps_k,\tau_k\), the sampled residual is below the safeguard
threshold, so RA-DCA solves the LP; the LP assigns weight to the flat piece and
keeps the iterate at the exact minimizer.

\begin{table}[!ht]
\centering
\caption{A near-active diagnostic for the LP fallback.  The exact residual is
computed using the exact active set, not the initial \(\eps_0\)-active set.
The LP branch prevents a spurious step caused by nearly active but exactly
inactive affine pieces.}
\label{tab:lpbranchrole}
\scriptsize
\begin{tabular}{lrrrrr}
\toprule
method & $x$ & objective & exact res. & iter. & LP calls\\
\midrule
centered eps-active & 0.01125 & 6.33e-05 & 0.0112 & 1 & 0\\
full-vertex eps-active & 0.02 & 0.0001 & 0 & 1 & 0\\
random eps-active & 0.0114 & 6.56e-05 & 0.0059 & 1 & 0\\
RA-DCA LP fallback & 0 & 0 & 0 & 1 & 1\\
RA-DCA vertex-forced & 0.02 & 0.0001 & 0 & 2 & 0\\
\bottomrule
\end{tabular}

\end{table}

The backend diagnostic solves random projected active-set LPs with dimensions
comparable to the sketches used in the experiments.  Gurobi solves the exact
\(\ell_\infty\) LP and is the default backend.  The projected fallback solves
the simplex-constrained least-squares surrogate used when neither Gurobi nor
\texttt{linprog} is available.  The column \(t/t_{\rm G}\) reports the maximum
projected residual relative to the Gurobi LP value.

\begin{table}[!ht]
\centering
\caption{Active LP backend diagnostic.  The full CSV file also reports
\texttt{linprog}; the table shows the exact Gurobi LP and the projected
simplex fallback.}
\label{tab:lpbackend}
\scriptsize
\begin{tabular}{rrlrrr}
\toprule
dirs & active & solver & mean ms & median ms & $t/t_{\rm G}$\\
\midrule
40 & 500 & gurobi & 12.81 & 12.54 & 1.00\\
40 & 500 & projected & 52.65 & 52.38 & 1.11\\
40 & 2000 & gurobi & 79.94 & 75.97 & 1.00\\
40 & 2000 & projected & 980.46 & 963.12 & 1.09\\
80 & 500 & gurobi & 20.03 & 19.98 & 1.00\\
80 & 500 & projected & 50.58 & 49.54 & 1.06\\
80 & 2000 & gurobi & 140.23 & 138.23 & 1.00\\
80 & 2000 & projected & 995.66 & 986.14 & 1.09\\
\bottomrule
\end{tabular}

\end{table}

The projected fallback gives residuals within about \(6\%\)--\(11\%\) of the
Gurobi LP value in these tests, although the simple MATLAB implementation is
slower for larger active sets.  The reported runs therefore use Gurobi when
available, while retaining the projected method as a portable fallback rather
than as the preferred high-performance backend.

\section*{Statements and Declarations}

\noindent\textit{Funding.}
The author was supported by the National Natural Science Foundation of China
(Grant No.~42450242) and the Beijing Overseas High-Level Talent Program.
Institutional support was provided by the Beijing Institute of Mathematical
Sciences and Applications (BIMSA).

\smallskip
\noindent\textit{Competing interests.}
The author has no relevant financial or non-financial interests to disclose.

\smallskip
\noindent\textit{Data and code availability.}
The accompanying reproducibility package contains the MATLAB source code,
input data, and scripts used to reproduce the reported experiments.  The
scripts regenerate the CSV summaries, LaTeX table fragments, and figure files
used by the manuscript.  The QUBO benchmark data are the OR-Library UBQP
instances \(bqp50\), \(bqp100\), and \(bqp250\), downloaded from the
\href{https://people.brunel.ac.uk/~mastjjb/jeb/orlib/files/}{OR-Library raw
data directory} as files \texttt{bqp50.txt}, \texttt{bqp100.txt}, and
\texttt{bqp250.txt}.  The best-known reference values are taken from the
\href{https://biqmac.aau.at/biqmaclib.html}{Biq Mac library}.  These values
are described in Section~\ref{sec:qubo}.  The real-data support-function,
top-\(k\) support, and trimmed sparse-regression experiments use the LIBSVM
data sets \texttt{a8a}, \texttt{phishing}, and \texttt{w8a}, downloaded from
the \href{https://www.csie.ntu.edu.tw/~cjlin/libsvmtools/datasets/binary/}
{LIBSVM binary data directory}.

\end{document}